\documentclass[12pt,a4paper]{amsart}
\usepackage[margin=20mm]{geometry}
\usepackage[numbers]{natbib}
\usepackage{hyperref}

\usepackage{amssymb,amsmath,bm}
\usepackage{amsthm}
\usepackage{bbm}
\usepackage{stmaryrd}
\usepackage[usenames,dvipsnames]{xcolor}
\usepackage{color}

\input xy
\xyoption{all} %\tolerance=500

\usepackage{amscd}
\usepackage{graphicx}
\usepackage{epsfig}
\usepackage[matrix,arrow,curve]{xy}
\usepackage{color}
\usepackage{mathrsfs}
\usepackage[scr=rsfs,cal=boondox]{mathalpha}

\usepackage{hyperref}
\usepackage{multirow}

\usepackage{latexsym}
\usepackage{amsfonts}
\input xy
\usepackage{tikz}
\usetikzlibrary{matrix}
\numberwithin{equation}{section}

\newtheorem{theorem}{Theorem}[section]
\newtheorem{corollary}[theorem]{Corollary}
\newtheorem{lemma}[theorem]{Lemma}
\newtheorem{proposition}[theorem]{Proposition}

\theoremstyle{definition}
\newtheorem{definition}[theorem]{Definition}
\newtheorem{remark}[theorem]{Remark}
\newtheorem{example}[theorem]{Example}

\newcommand{\Id}{\mathbbmss{1}}

\DeclareMathOperator{\GL}{GL}
\DeclareMathOperator{\Hom}{Hom}

\DeclareMathOperator{\Lin}{Lin}

\newcommand{\catname}[1]{\textnormal{\texttt{#1}}}

\mathchardef\za="710B  %\alpha
\mathchardef\zb="710C  %\beta
\mathchardef\zg="710D  %\gamma
\mathchardef\zd="710E  %\delta
\mathchardef\zve="710F %\epsilon
\mathchardef\zz="7110  %\zeta
\mathchardef\zh="7111  %\eta
\mathchardef\zvy="7112 %\theta
\mathchardef\zi="7113  %\iota
\mathchardef\zk="7114  %\kappa
\mathchardef\zl="7115  %\lambda
\mathchardef\zm="7116  %\mu
\mathchardef\zn="7117  %\nu
\mathchardef\zx="7118  %\xi
\mathchardef\zp="7119  %\pi
\mathchardef\zr="711A  %\rho
\mathchardef\zs="711B  %\sigma
\mathchardef\zt="711C  %\tau
\mathchardef\zu="711D  %\upsilon
\mathchardef\zvf="711E %\phi
\mathchardef\zq="711F  %\chi
\mathchardef\zc="7120  %\psi
\mathchardef\zw="7121  %\omega
\mathchardef\ze="7122  %\varepsilon
\mathchardef\zy="7123  %\vartheta
\mathchardef\zf="7124  %\varomega
\mathchardef\zvr="7125 %\varrho
\mathchardef\zvs="7126 %\varsigma
\mathchardef\zf="7127  %\varphi
\mathchardef\zG="7000  %\Gamma
\mathchardef\zD="7001  %\Delta
\mathchardef\zY="7002  %\Theta
\mathchardef\zL="7003  %\Lambda
\mathchardef\zX="7004  %\Xi
\mathchardef\zP="7005  %\Pi
\mathchardef\zS="7006  %\Sigma
\mathchardef\zU="7007  %\Upsilon
\mathchardef\zF="7008  %\Phi
\mathchardef\zW="700A  %\Omega

\newcommand{\be}{\begin{equation}}
\newcommand{\ee}{\end{equation}}
\newcommand{\ra}{\rightarrow}

\newcommand{\bea}{\begin{eqnarray}}
\newcommand{\eea}{\end{eqnarray}}
\newcommand{\beas}{\begin{eqnarray*}}
\newcommand{\eeas}{\end{eqnarray*}}
\def\*{{\textstyle *}}

\newcommand{\nn}{\nonumber}

\newcommand{\N}{\mathbb{N}}
\newcommand{\Z}{\mathbb{Z}}
\newcommand{\R}{\mathbb{R}}

\newcommand{\pa}{\partial}
\newcommand{\ti}{\times}

\def\tU{\tilde U}

\def\Hom{\mathsf{Hom}}

\def\bp{{\boxplus}}

\def\cJ{{\mathcal J}}

\def\cH{{\mathcal H}}
\def\cO{{\mathcal O}}

\def\cA{\mathcal{A}}

\def\Sec{\operatorname{Sec}}

\def\Lin{\mathsf{Lin}}

\def\bp{{\mathbf p}}

\def\bq{{\mathbf q}}

\def\bu{{\mathbf u}}

\def\sT{{\mathsf T}}

\def\xd{\mathrm{d}}
\def\xi{\tx{i}}

\def\dt{\xd_{\sT}}

\newdir{|>}{%
!/4.5pt/@{|}*:(1,-.2)@^{>}*:(1,+.2)@_{>}}

\def\dim{\operatorname{dim}}

\def\GL{\operatorname{GL}}

\def\Id{\operatorname{Id}}

\def\tx{\tilde{x}}
\def\tzx{\tilde{\zx}}
\def\tv{{\tilde{v}}}
\newdir{ (}{{}*!/-5pt/@^{(}}
\newcommand{\id}{\mathrm{id}}

\def\bza{{\boldsymbol\za}}
\def\n{\nabla}

\newcommand{\mn}{{\medskip\noindent}}

\newcommand{\no}{{\noindent}}

\def\z2{{\Z_2^n}}

\def\Fr{\operatorname{Fr}(E)}
\title[Principal bundles in the category of $\Z_2^n$-manifolds]{Principal bundles\\ in the category of $\Z_2^n$-manifolds}

 \author{Andrew James Bruce \& Janusz Grabowski}
\date{\today}
 
\begin{document}

\begin{abstract}
We introduce and examine the notion of principal $\Z_2^n$-bundles, i.e., principal bundles in the category of $\Z_2^n$-manifolds. The latter are higher graded extensions of supermanifolds in which a $\z2$-grading replaces $\Z_2$-grading. These extensions have opened up new areas of research of great interest
in both physics and mathematics. In principle, the geometry of $\z2$-manifolds is essentially different than that of supermanifolds, as for $n>1$ we have formal variables of even parity, so local smooth functions are formal power series. On the other hand, a full version of differential calculus is still valid. We show in this paper that the fundamental properties of classical principal bundles can be generalised to the setting of  this `higher graded' geometry, with properly defined frame bundles of $\z2$-vector bundles as canonical examples. However, formulating these concepts and proving these results relies on many technical upshots established in earlier papers. A comprehensive introduction to $\z2$-manifolds is therefore included together with basic examples.
\par

\bigskip\no{\bf Keywords:};
 supermanifold; graded manifold, principal bundle; supergroup action; functor of points.\par

\mn{\bf MSC 2020:}~14A22;~18F15;~58A50;~58D19.
\end{abstract}

\maketitle

\setcounter{tocdepth}{2}

\tableofcontents
%\newpage
\section{Introduction}
It is well known that principal bundles have many applications in topology, differential geometry, and mathematical physics. In particular, gauge fields are, globally, connections on a principal bundle. Principal superbundles have been studied in the different approaches to supermanifolds, though there are relatively few comprehensive works in the literature (for the locally ringed space approach see \cite{Bartocci:1991, Carmeli:2018,Kessler:2019,Kruglikov:2021,Stavracou:1998}). We remark that principal superbundles were applied to supergauge theories quite early on in the development of supergeometry (see \cite{Almorox:1987}). It should be noted that field configurations in (pure) supergravity (ignoring potential auxiliary fields) are Cartan connections locally given by one-forms over space-time with values in the Poincar\'{e} Lie superalgebra. The relation between Cartan geometry and supergravity is somewhat ``old news'', but was made very explicit by Edler \cite{Edler:2020} in 2023. However, in the physics literature, it is often the case that global considerations, such as non-trivial principal bundles, are not explored. A discussion of supergauge theory and supergravity is outside the scope of this paper, we point out that these concepts are mathematically related to principal superbundles. Presumably, ``higher graded'' gauge theory would rely on similar constructions. It should be noted that frame bundles are important examples of principal bundles. In general relativity,  local Lorentz transformations and general coordinate transformations are naturally dealt with using frame bundles. Indeed, part of the motivation for this paper was to understand frame bundles in the setting of $\Z_2^n$-manifolds. Irrespective of possible links with physics, developing generalisations of principal superbundles is a mathematically interesting pursuit with many subtleties.     \par
In this paper, we discuss the foundational elements of the theory of principal bundle structures in the category of $\Z_2^n$-manifolds: a novel and non-trivial generalisation of supermanifolds (see \cite{Covolo:2016,Covolo:2016a,Covolo:2016b,Covolo:2016c}). Our definition of a principal bundle uses local trivialisations and $\Z_2^n$-Lie group actions that preserve the fibres (see  Definition \ref{def:PrinBun}). $\Z_2^n$-Lie groups and their actions are a subject that has barely been explored (see \cite{Bruce:2020a, Bruce:2020b} for the first works in this area).  We use the functor of points to define $\Z_2^n$-Lie groups and their actions on $\Z_2^n$-manifolds. Alternatively, one could develop a Hopf algebra-based approach, however, the functor of points approach allows us to mimic the classical setting quite closely. This gives a ``clean'' formulation that differential geometers and physicists can follow without the need for Hopf algebras and similar. Recall that the category of Lie supergroups is equivalent to the category of super Harish--Chandra pairs (loosely, a Lie group and a Lie superalgebra with some compatibility conditions, see \cite{Carmeli:2011} and references therein). We refrain from developing the  $\Z_2^n$ analogs of super Harish--Chandra pairs, though, presumably, this can be done.  We remark that under the umbrella of noncommutative geometry, Hopf--Galois extensions are seen as a noncommutative analogue of a principal bundle (see [Chapter 5] \cite{Beggs:2019}). Other approaches have been developed, including coalgebra bundles (see \cite{Brzezinski:1998}). \par

Historically,  principal bundles were defined via a free proper action on a manifold, the local triviality coming as a consequence. In the context of supermanifolds and $\Z_2^n$-manifolds, a map is proper if the underlying map of topological spaces is proper. However, we will take the fibre bundle structure as part of our initial definition. \par
We show, with a little care, that the well-known constructions and results in the theory of classical principal bundles generalise to the setting of $\Z_2^n$-geometry. The results are quite similar to the supermanifold case, but the proofs require delicate handling compared with the classical proofs.   We denote a principal $\Z_2^n$-bundle with base $M$ and group structure $G$ as $P(M,G)$.  Our main results are:
\begin{enumerate}
\itemsep1em
\item Underlying a principal $\Z_2^n$-bundle is a classical principal bundle (Proposition \ref{prop:RedPrinBun}),
\item Given atlases of $M$ and $G$, one can construct an atlas of $P(M,G)$ (Proposition \ref{prop:Atlas}),
\item The associated orbit space $P\slash G$ is $M$ (Proposition \ref{prop:OrbitSpace}),
\item Fixing $M$ and $G$, all morphisms of such principal $\Z_2^n$-bundle are isomorphisms (Proposition \ref{prop:AllIso}),
\item A principal $\Z_2^n$-bundle is trivial if and only if it admits a global section (Proposition \ref{prop:GlobSecTriv}),
\item $P(M,G)$ can be defined via a collection of maps that satisfy the cocycle-condition (Theorem \ref{thm:CocyGlue}),
\item We use the results and constructions found earlier in this paper to build an original concept of a $\z2$-vector bundle and its frame bundle in Section \ref{sec:Fbundles}.
\end{enumerate}
\smallskip

\noindent \textbf{Arrangement.} In Section \ref{sec:Z2nGeom} we review the theory of $\Z_2^n$-manifolds, including the functor of points as needed in the next section. We proceed with the main definitions and results of this paper in Section \ref{sec:PrinBun}. We define $\z2$-vector bundles and the corresponding frame bundles as canonical examples of principal $\z2$-bundles in Section \ref{sec:Fbundles}. We end in Section \ref{sec:Conc} with some concluding remarks.
\section{Rudiments of $\Z_2^n$-geometry}\label{sec:Z2nGeom}
\subsection{$\Z_2^n$-manifolds as locally ringed spaces}
We will assume some familiarity with supergeometry and $\Z_2^n$-geometry. There are no new results in the following subsections, rather we recap the essential results needed.  To set notation, by $\Z_2^n$ we mean the abelian group $\Z_2 \times \Z_2 \times \cdots \times \Z_2$ where the Cartesian product is over $n$ factors.  Once an ordering has been fixed, we will denote elements of $\Z_2^n$ as $\gamma_i$ for $i = 0,1 \cdots , N$, where  $N = 2^n-1$. We will set $ \mathbf{0} := \gamma_0 = (0,0, \cdots, 0)$ for later convenience. For the ordering of elements of  $\Z_2^n$, we will fill in the zeros and the ones from the left,  and then place the even elements first and then the odd elements.\par
A \emph{$\Z_2^n$-manifold} is understood in the sense of Covolo--Grabowski--Poncin  \cite{Covolo:2016,Covolo:2016a}, i.e,  as a locally ringed space $M =(|M|, \cO_M)$ such that for `small enough' open subsets $|U| \subset |M|$, $\cO_M(|U|) \cong C^\infty(|U|)[[\zx]]$. Here $\zx^\alpha$ are  the \emph{formal variables} that are assigned a non-zero $\Z_2^n$-degree and are subject to the relation
$$ \zx^\alpha \, \zx^\beta = (-1)^{\langle \deg(\alpha) , \deg(\beta)\rangle}\, \zx^\beta \, \zx^\alpha\,,$$
where $\deg(\alpha) :=  \deg(\zx^\alpha) \in  \Z_2^n \setminus \{\mathbf{0} \}$ and
$${\langle \deg(\alpha) , \deg(\beta)\rangle}=\sum_i\deg(\alpha)_i\deg(\zb)_i.$$
We set $\dim(M) = p|\mathbf{q}$, where $\mathbf{q} = (q_1, q_2, \cdots , q_N)$ where $p := \dim(|M|)$ and we have $q_i$ formal variables of $\Z_2^n$-degree $\gamma_i$ ($i >0$). The simplest example is $\R^{p|\mathbf{q}}$, where $|M|=\R^p$. A formal variable $\zx^\za$ we call \emph{odd} if $\big(\zx^\za\big)^2=0$, i.e., $|\za|={\langle \deg(\alpha) , \deg(\alpha)\rangle}$ is odd, and \emph{even} in the other case. Note that, in contrast with the standard $\Z_2$-case, for $n>1$ even formal coordinates do not generally commute with the others and we are dealing with formal power series in the even variables. This is crucial for obtaining a consistent differential calculus on $\Z_2^n$-manifolds. \par
An \emph{open $\Z_2^n$-submanifold} of $M$ is a $\Z_2^n$-manifold of the form $U = (|U|, \cO_M|_{|U|})$, with $|U| \subset |M|$ is open. The underlying manifold $|M|$ we will refer to as the \emph{reduced manifold}.  \emph{Morphisms of $\Z_2^n$-manifolds} are morphisms of locally ringed spaces. That is, a morphism $\phi : M \rightarrow N$ consists of a pair $ \phi = (|\phi|, \phi^*  )$, where $|\phi| : |M| \rightarrow |N|$ is a continuous map (in fact, smooth) and  $\phi^*$  is a family of $\Z_2^n$-graded $\Z_2^n$-commutative ring morphisms $\phi^*_{|V|} : \cO_N(|V|) \rightarrow \cO_M\big( |\phi|^{-1}(|V|)\big)$, that respect the restriction maps for every open $|V| \subset |N|$. The category of $\Z_2^n$-manifolds we will denote as $\Z_2^n\catname{Man}$. \par
The \emph{ideal sheaf} $\mathcal{J}_M$ is defined by
$\mathcal{J}(|U| ) := \langle f \in \cO_M(|U|)~|~  f~\textnormal{is of non-zero}~ \Z_2^n \textnormal{degree} \rangle$. In other words,
$$C^\infty_{|U|}=(\cO_M(|U|)/\mathcal{J}(|U|).$$
The $\mathcal{J}_M$-adic topology on $\cO_M$ can then be defined in the obvious way. Note that the ideal sheaf is not nilpotent, as it is for standard supermanifolds. However, the structure sheaf of a $\Z_2^n$-manifold $\cO_M$ is Hausdorff complete with respect to the $\mathcal{J}_M$-adic topology (see \cite[Proposition 7.9]{Covolo:2016} for details). In particular, we have a short exact sequence of sheaves of $\Z_2^n$-graded $\Z_2^n$-commutative associative $\R$-algebras
\begin{equation}\label{eqn:SES}
0\longrightarrow\ker\zve\longrightarrow\cO_M\stackrel{\zve}{\longrightarrow}C^\infty_{|M|}\longrightarrow 0\,,
\end{equation}
where $\ker\zve=\cJ_M$. We thus have a canonical morphisms of $\Z_2^n$-manifolds $\iota : |M| := (|M|, C^\infty_{|M|}) \rightarrow (|M|,\cO_M)=: M$. While morphism $M \rightarrow |M|$ do exist,  none are canonical. Thus, we have a Batchelor-Gaw{\c e}dzki (\cite{Gawedzki:1977}) type theorem for $\Z_2^n$-manifolds (see \cite{Covolo:2016a}).
\subsection{Atlases and coordinates}
Due to the local structure of a  $\Z_2^n$-manifold  of dimension $p|\mathbf{q}$ we know that for a `small enough' $|U| \subset |M|$ there exists an isomorphism of $\Z_2^n$-manifolds
$$\psi : U \longrightarrow  \mathcal{U}^{p|\mathbf{q}} = (\mathcal{U}^p , C^\infty_{\mathcal{U}^p}[[\zx]])\,,$$
where $\mathcal{U}^p \subset \R^p$ is open and $\mathbf{q} = (q_1, q_2, \cdots , q_{2^n-1})$. This allows us to employ local coordinates $x^A:=(x^a, \zx^\alpha)$, where $x^a$ form a coordinate system on $\mathcal{U}^p$, where $p := \dim(|M|)$, and $\zx^\za$ are formal coordinates with degrees in $\Z^n_2$ such that we have $q_i$ formal variables of $\Z_2^n$-degree $\gamma_i$ ($i >0$). A pair $(U, \psi)$ we refer to as a \emph{(coordinate) chart}, and a family of charts $\{(U_i , \psi_i)  \}_{i \in \mathcal{I}}$ we refer to as an \emph{atlas}  if the family $\{ |U_i| \}_{i \in \mathcal{I}}$ forms an open cover of $|M|$.  Coordinate transformations are constructed as follows. Suppose we have two charts
\begin{align*}
&\psi_i :  U_i \stackrel{\sim}{\rightarrow} \mathcal{U}_i^{p|\mathbf{q}}= \big(\mathcal{U}^p_i, C^\infty_{\mathcal{U}^p_i}[[\zx]] \big), && \psi_j :  U_j \stackrel{\sim}{\rightarrow} \mathcal{U}_j^{p|\mathbf{q}}= \big(\mathcal{U}^p_j, C^\infty_{\mathcal{U}^p_j}[[\eta]]\big),
\end{align*}
where we assume that $|U_{ij}|:= |U_i|\cap |U_j| \neq \emptyset $. We then set
\begin{align*}
& |\psi_i|\big( |U_{ij}|\big):= \mathcal{U}^p_{ij}, &&|\psi_j|\big( |U_{ij}|\big):= \mathcal{U}^p_{ji}.
\end{align*}
As standard, we have a homomorphism (not writing out the obvious restrictions)
$$|\psi_{ij}| := |\psi_j|\circ |\psi_i|^{-1} : \mathcal{U}^p_{ij} \longrightarrow \mathcal{U}^p_{ji}.$$
We need to define what happens to sections, and via the Chart Theorem, we know it is sufficient to describe this in terms of the coordinates.  Specifically,
$$\psi^*_{ij} :=  (\psi_i^*)^{-1} \circ \psi_j^* :  C^\infty(\mathcal{U}^p_{ji})[[\eta]] \longrightarrow C^\infty(\mathcal{U}^p_{ji})[[\zx]],$$
provides the required $\Z_2^n$-graded  $\Z_2^n$-commutative algebra isomorphism.   In coordinates, we write $\psi^*_{ij}(y^b ,\eta^B) =  (y^b(x, \zx),  \eta^B(x, \zx))$ and similar. Thus we have a family of $\Z_2^n$-manifold morphisms
$$\psi_{ij} : \mathcal{U}^{p|\mathbf{q}}_{ij} \rightarrow \mathcal{U}^{p|\mathbf{q}}_{ji}\,,$$
that satisfy the standard gluing conditions.
\begin{example}
The $\Z_2^2$-circle is defined as $\mathbb{S}^{1|1,1,1} =\mathbb{S}^1 \times \R^{0|1,1,1}$. We proceed to construct an atlas as follows. First, we build a chart on the circle $\mathbb{S}^1$ as
\begin{align*}
&U_1 =  \mathbb{S}^1 - \{(1,0)\},  &&U_2 =  \mathbb{S}^1 - \{(-1,0)\}\\
&\phi_1(x,y) = \arctan(y/x), \quad 0 < \phi_1(x,y)< 2 \pi && \phi_2(x,y) = \arctan(y/x),\quad  -\pi < \phi_1(x,y)<  \pi
\end{align*}
The inverses are $\phi^{-1}_{1/2}(\theta) =  (\cos \theta, \sin \theta)$. The transition functions can easily be seen to be
\begin{align*}
&\psi^*_{21}\theta_2 =
\begin{cases}
\theta_1 & \text{if } \theta_2 > 0 \\
\theta_1- 2\pi & \text{if } \theta_2 < 0
\end{cases}
&&
\psi^*_{12}\theta_1 =
\begin{cases}
\theta_2 & \text{if } \theta_1 < \pi \\
\theta_1+ 2\pi & \text{if } \theta_1 > \pi
\end{cases}
\end{align*}
We now need to define formal coordinates on each of the opens and describe the coordinate changes. We define the coordinates $(z_i, \eta_i, \chi_i)$ over $U_i$ (where $i = 1,2$) of degrees $(1,1), (0,1)$ and $(1,0)$, respectively. As we have a Cartesian product, the coordinate transformations are trivial, i.e.,
\begin{align*}
\psi^*_{21}(z_2, \eta_2, \chi_2) = (z_1, \eta_1, \chi_1), &\qquad
\psi^*_{12}(z_1, \eta_1, \chi_1) = (z_2, \eta_2, \chi_2)\,.
\end{align*}
\end{example}
Much like the case for manifolds, one can construct global geometric objects by gluing local geometric objects. That is, we can consider a $\Z_2^n$-manifold as being covered by  $\Z_2^n$-domains together with the appropriate gluing information, i.e., admissible coordinate transformations.   Importantly, we have the \emph{Chart Theorem} (\cite[Theorem 7.10]{Covolo:2016}) that allows us to write morphisms of $\Z_2^n$-manifolds in terms of the local coordinates.  Specifically, suppose we have two  $\Z_2^n$-domains $\mathcal{U}^{p|\mathbf{q}}$ and $\mathcal{V}^{r|\mathbf{s}}$. Then morphisms $\phi: \mathcal{U}^{p|\mathbf{q}} \longrightarrow \mathcal{V}^{r|\mathbf{s}}$  correspond  to \emph{graded unital $\R$-algebra morphisms}
 \begin{equation*}
 \phi^* : C^{\infty}\big(\mathcal{V}^r \big)[[\eta]] \longrightarrow   C^{\infty}\big(\mathcal{U}^p \big)[[\zx]],
 \end{equation*}
which are fully determined by their coordinate expressions. Note that such algebra morphisms, by definition, preserve the $\Z_2^n$-degree.

\subsection{Local differential calculus}
Using local coordinates $(x^a, \zx^\alpha)$, a function (section of the structure sheaf) is of the form (using multi-index notation)
\begin{equation}
f(x, \zx) = \sum_{\bm{\alpha}\in \N^{\times |\bm{q}|}}  \, \zx^{\bm \alpha} f_{\bm \alpha}(x)\,,
\end{equation}
where $\bm{q} = (q_1, q_2, \cdots, q_N) \in \N^{2^n-1}$, and $|\bm{q}| = \sum_{i=1}^{2^n-1}q_i$. The components $f_{\bm \alpha}(x)$ are smooth functions in the degree $\bm 0$ coordinates.  Note, that we have a formal power series in even formal coordinates, while the odd formal coordinates are nilpotent. \par
Partial derivatives with respect to the degree $\bm 0$ coordinates are defined as standard using limits. Note
$$\frac{\partial}{\partial x^a}(\zx^\alpha \mbox{---}) = \zx^\alpha \frac{\partial}{\partial x^a}(\mbox{---})\,.$$
Partial derivatives with respect to the formal coordinates are defined algebraically:
\begin{equation}
\frac{\partial \zx^\beta}{\partial \zx^\alpha}= \delta^\beta_\alpha\, ,
\end{equation}
which is extended to formal power series via
\begin{align}
\frac{\partial}{\partial \zx^\alpha}(\zx^{\beta_1}\zx^{\beta_2}\cdots \zx^{\beta_p}) &= \delta^{\beta_1}_\alpha\zx^{\beta_2}\cdots \zx^{\beta_p} +
(-1)^{\langle \deg(\alpha), \deg(\beta_1)\rangle}\, \zx^{\beta_1}\delta^{\beta_2}_\alpha\cdots \zx^{\beta_p} \\
\nonumber &+ \cdots (-1)^{\langle \deg(\alpha), \deg(\beta_1)+ \deg(\beta_2) + \cdots \deg(\beta_{p-1})\rangle}\, \zx^{\beta_1}\zx^{\beta_2}\cdots \delta^{\beta_p}_\alpha\,.
\end{align}
Setting $x^A = (x^a, \zx^\alpha)$, we have
$$\frac{\partial}{\partial x^A} \frac{\partial}{\partial x^B} =(-1)^{\langle\deg(A), \deg(B) \rangle}\,\frac{\partial}{\partial x^B} \frac{\partial}{\partial x^A}\,. $$
Clearly, for odd formal coordinates $\left(\frac{\partial}{\partial \zx}\right)^2 = 0$.
\subsection{The Cartesian product}
It is known that the category $\Z_2^n\catname{Man}$ admits all finite products, for details the reader should consult \cite{Bruce:2018,Bruce:2018b}. One can construct an atlas on $M \times N$ using atlases on $M$ and $N$ in more or less the same way as one can in the category of supermanifolds.\par
Let $M = (|M|, \cO_M)$ and $N = (|N|,\cO_N)$ be two $\Z_2^n$-manifolds of dimension $p|\textbf{q}$ and $r|\textbf{s}$, respectively.  The Cartesian products $|U|\times |V|$, where $|U|\in |M|$, $|V|\in |N|$ are open, form a basis of the product topology of $|M|\times |N|$. If we have coordinate charts $\{(U_i, \psi_i) \}_{i \in \mathcal{I}}$ and $\{(V_j, \phi_j)\}_{j \in \mathcal{J}}$, with coordinates $(x_i, \zx_i)$ and $(y_j, \eta_j)$ on $M$ and $N$, respectively, then we can define the Cartesian product as follows.  First, we denote the basis of the product topology generated by $|U_i|\times |V_i|$ (here both the underlying topological spaces of the coordinate charts) as $\mathcal{B}$. Recall that working with a basis of a topology rather than all opens is sufficient for the gluing of sheaves.
\begin{definition}
Let $M$ and $N$ be $\Z_2^n$-manifolds of dimension $p|\mathbf{q}$ and $r|\mathbf{s}$, respectively. The \emph{product $\Z_2^n$-manifold} $M\times N$, of dimension $p+r| \mathbf{q}+\mathbf{s}$, is the locally $\Z_2^n$-ringed space $(|M|\times |N|, \cO_{M\times N})$, where $|M|\times |N|$ is equipped with the product topology, and $\cO_{M\times N}$ is the sheaf glued from the sheaves $C^\infty_{|U_i|\times|V_j|}(x_i , y_j)[[\zx_i, \eta_j]]$ associated to the base $\mathcal{B}$, i.e.,
$$\cO_{M\times N}|_{|U_i|\times  |V_j|} \simeq C^\infty_{|U_i|\times|V_j|}(x_i , y_j)[[\zx_i, \eta_j]]\,.$$
\end{definition}
It was shown in \cite{Bruce:2018b} that the product $\Z_2^n$-manifold is a $\Z_2^n$-manifold and satisfies all the required properties to be a categorical product. \par
We also have the \emph{fundamental isomorphism} (recall that the algebras in question here are nuclear Fr\'echet)
\begin{equation}
C^\infty(|U|\times |V|)[[\zx,\eta]] \simeq C^\infty(|U|)[[\zx]]\widehat{\otimes}_\R C^\infty(|V|)[[\eta]]\,,
\end{equation}
where the completion is taken with respect to any locally convex topology on the tensor product $ C^\infty(|U|)[[\zx]] \otimes_\R C^\infty(|V|)[[\eta]]$.
We construct the projection map $\pi_M : M\times N \rightarrow M$, as the $\Z_2^n$-manifold morphism whose underlying continuous map $|\pi_M| =: \pi_{|M|}:  |M|\times |N| \rightarrow |M|$ is the standard smooth projection, and  using the basis $\mathcal{B}$,
$$\pi_M^*|_{|U|} : \cO_M(|U|)  \rightarrow  \cO_M(|U|)\widehat{\otimes}_\R \cO_M(|V|)\,,$$
is the inclusion of (topological) algebras.
\subsection{The tangent space and surjective submersions}
We will draw heavily on \cite{Covolo:2021} in this subsection. The \emph{tangent space} $\sT_m M$ at $m \in |M|$ is the $\Z_2^n$-graded $\R$-vector space of $\Z_2^n$-graded $\R$-linear derivations $\delta :  \cO_{M, m}\rightarrow \R$. It is clear that if $\dim(M) = p|\mathbf{q}$ then $\dim_R(\sT_m M) =  p|\mathbf{q}$. Moreover, given a coordinate system $(x^a, \zx^\alpha)$ centred on $m \in |M|$, we have a basis of  $\sT_m M$ given by $(\partial_{x^a}|_ m, \partial_{\zx^\alpha}|_ m)$. Any $\Z_2^n$-graded derivation $X : \cO_M(|U|) \rightarrow \cO_M(|U|)$ induces a $\Z_2^n$-graded derivation at the level of stalks $X|_{m} : \cO_{M,\, m} \rightarrow \cO_{M,\, m}$, of the same $\Z_2^n$-degree as $X$ (assuming homogeneity), for any $m \in |U|$. We denote $\epsilon_{m} : \cO_{M,\, m} \rightarrow C^\infty_{m}$ as the algebra morphism induced by the pullback to the reduced manifold of $M$. Furthermore the evaluation morphism at $m$ we denote as $\textrm{ev}_{m} : C^\infty_{m} \rightarrow \R$. We then define $X_{m} \in \sT_{m}M$ as the tangent vector
$$X_{m} := (\textrm{ev}_{m} \circ \epsilon_{m} \circ X|_{m}). $$
\begin{definition}
Let $\phi : M \rightarrow N$ be a morphism of $\Z_2^n$-manifolds such that $|\phi|(m) =n$ for some $m \in |M|$. The \emph{tangent map} of $\phi$ at $m$ is the  degree zero  morphism of $\Z_2^n$-vector spaces
 $$(\sT \phi)_m : \sT_m M\rightarrow \sT_n N$$
given by
$$(\sT \phi)_m(v):= v \circ \phi^*_m \,,$$
for all  $v \in \sT_m M$. Moreover, $\phi$ is a \emph{surjective submersion} if for all points $m \in |M|$ the tangent map $\sT \phi_m$ is surjective.
\end{definition}
Using coordinates $(x^a, \zx^\alpha)$ and $(y^b, \theta^\beta)$ on $M$ and $N$, respectively, we get
$$(\sT \phi)_m  \simeq \begin{pmatrix}
\textrm{ev}_m \epsilon (\partial_x y) & \textrm{ev}_m \epsilon (\partial_\zx y) \\
\textrm{ev}_m \epsilon (\partial_x \theta) & \textrm{ev}_m \epsilon (\partial_\zx \theta)
\end{pmatrix} = \begin{pmatrix}
\textrm{ev}_m \epsilon (\partial_x y) & 0 \\
0 & \textrm{ev}_m \epsilon (\partial_\zx \theta)
\end{pmatrix}\, ,$$
which is a diagonal real-valued matrix. The following lemma is thus evident.
\begin{lemma}\label{lem:SurSub}
If $\pi =(|\pi|, \pi^*) : M \rightarrow N$ is a surjective submersion, then $|\pi| : |M| \rightarrow |N|$ is a surjective submersion of smooth manifolds.
\end{lemma}
Due to \cite[Proposition 4.3]{Covolo:2021}, we always have adapted coordinates $(x^A, y^B)$ (of all degrees) such that, symbolically, the surjective submersion is locally the projection $(x^A, y^B) \mapsto x^A$. More carefully, let $\phi : M \rightarrow N$ be a morphism of $\Z_2 ^n$-manifolds, where the dimension of $M$ is $p|\mathbf{q}$ and the dimension of $N$ is $r|\mathbf{s}$. We require $p \geq r$ and $q_i \geq s_i$ for all $i$. Furthermore, let us fix some arbitrary point $m \in |M|$. Then $\phi$ is a submersion at $m$  if and only if there exists charts $(U, \psi)$ with coordinates $(x, \zx)$ centred on $m$ and $(V, \chi)$ with coordinates $(y, \eta)$ centred on $\phi(m)\in N$, such that $\phi|_{|U|}: U \rightarrow V$ has the form
\begin{align*}
\phi^*(y^a) = x^a  \qquad 1\leq a  \leq r\,, && \phi^*(\eta^\alpha_\gamma) = \zx^\alpha_\gamma  \qquad 1\leq \alpha \leq s_\gamma\,,
\end{align*}
for all $\gamma \in \Z_2^n\setminus \{\mathbf{0}\}$.
\begin{remark}
The $\Z_2^n$ analogues of the inverse function theorem, the implicit function theorem, and the constant rank theorem were presented and proved in   \cite{Covolo:2021}.
\end{remark}
\subsection{The functor of points}
In this subsection we rely heavily on results presented in \cite{Bruce:2020a,Bruce:2018b}. Grothendieck's functorial approach to algebraic geometry (see \cite{Grothendieck:1973a}) has proven extremely useful in supergeometry and is essential in later sections of this paper. For any $S = (|S|, \cO_S) \in \Z_2^n\catname{Man}$, the set of \emph{$S$-points on $M$} is defined as
$$M(S) := \Hom_{\Z_2^n\catname{Man}}(S, M) \,.$$
We have a contravariant functor $M(-) : \Z_2^n\catname{Man}^{\textrm{op}} \rightarrow \catname{Set}$ from the category of $\Z_2^n$-manifolds to sets which fully determines $M$. Note that not all such functors are representable and we will refer to non-representable functors as \emph{generalised $\Z_2^n$-manifolds}. The fundamental result here is \emph{Yoneda's lemma}, which states that any morphism $\phi :  M \rightarrow N$ defines a natural transformation $\phi_- : M(-) \rightarrow N(-)$ and any natural transformation between $M(-)$ and $N(-)$ arises from a unique morphism between the $\Z_2^n$-manifolds.  A corollary of Yoneda's lemma is that two $\Z_2^n$-manifolds are diffeomorphic if and only if $M(-)$ and $N(-)$ are naturally isomorphic, i.e., $M(S) \cong N(S)$ for all $S\in \Z_2^n\catname{Man}$.\par
A fundamental and important result here is \cite[Theorem 9]{Bruce:2018b}, which allows us to identify morphisms between $\Z_2^n$-manifolds $\phi : M \rightarrow N$ with morphisms of (unital) $\Z_2^n$-graded $\Z_2^n$-commutative algebras $\phi^* : \cO_N(|N|) \rightarrow \cO_M(|M|)$. Thus, in terms of the functor of points, we have the useful result
$$M(S) := \Hom_{\Z_2^n\catname{Man}}(S, M) \cong \Hom_{\Z_2^n\catname{Alg}}(\cO_M(|M|), \cO_S(|S|))\,.$$
Here $\Z_2^n\catname{Alg}$ is the category of $\Z_2^n$-graded associative algebras. Note that, by definition morphisms in this category preserve the $\Z_2^n$-degree. Moreover, if the algebras are unital, as all the algebras we encounter in this paper are, we insist that morphisms respect the units. \par
Given any point $m \in |M|$ we can consider it as a morphism of $\Z_2^n$-manifolds $m := (|m|, m^*): \R^{0|\mathrm{0}} \rightarrow M$ with $|m| :  \R^0 \rightarrow |M|$ being the assignment of a single element $\star$ to the point $m \in |M|$.  The pullback is the evaluation at $m$, i.e., $m^* := \textrm{ev}_m \circ \epsilon_m : \cO_{M, m} \rightarrow \R$, where $\textrm{ev}_m:  C^\infty_m \rightarrow \R$ is the classical evaluation map, and $\epsilon_m : \cO_{M,m}\rightarrow C_m^\infty$ is the map induced by the sheaf morphism $\epsilon : \cO_M \rightarrow C^\infty_{|M|}$. It can be shown that $\R^{0|\mathbf{0}}$ is a terminal object in $\Z_2^n\catname{Man}$, that is, for every $S\in \Z_2^n\catname{Man}$ there is a unique map $!_S :  S \rightarrow \R^{0|\mathbf{0}}$. This, of course induces a map $M(!_S): M(\R^{0|\mathbf{0}}) \rightarrow M(S)$ defined as $m \mapsto m_S := m \, \circ \,!_S$. \par
 It is known that $\Z_2^n$-manifolds whose underlying topological space is just a single point, so $\Z_2^n$-manifolds of the form $\R^{0|\mathbf{q}} = (\star, \R[[\zx]])$, are sufficient for the functor of points. We refer to such $\Z_2^n$-manifolds as \emph{$\Z_2^n$-points}, the category of which we denote as $\Z_2^n\catname{Pts}$, (see \cite[Theorem 3.8]{Bruce:2020a}).\par
We set $\Lambda := \R[[\zx]]$ and refer to such algebras as $\Z_2^n$-Grassmann algebras and denote the corresponding category as $\Z_2^n\catname{GrAlg}$.  We have an equivalence of categories between $\Z_2^n\catname{GrAlg}$  and $\Z_2^n\catname{Pts}^{\textrm{op}}$. More technically, we have the functor
\begin{align*}
\mathcal{P} :& ~  \Z_2^n\catname{GrAlg} \longrightarrow \Z_2^n\catname{Pts}^{\textrm{op}}\\
 & \mathcal{P}(\Lambda) :=  (\star, \Lambda),
\end{align*}
and given a morphism $\varphi^* : \Lambda \rightarrow \Lambda'$ we set $\mathcal{P}(\varphi^*) := (|\varphi|, \varphi^*)$, where $|\varphi|$ is the identification of the underlying points of $\mathcal{P}(\Lambda)$ and $\mathcal{P}(\Lambda')$. The $\Lambda$-points of a $\Z_2^n$-manifold $M$ is defined as
 $$M(\Lambda) := M(\mathcal{P}(\Lambda)) = \Hom_{\Z_2^n\catname{Man}}(\mathcal{P}, M) \simeq \Hom_{\Z_2^n\catname{Alg}}(\Lambda, \cO_M(|M|)\,.$$
Consider an open cover $\{ |U_i|\}_{i \in \mathcal{I}}$ of $|M|$ and the open $\Z_2^n$-submanifolds $U_i = (|U_i|, \cO_M|_{|U_i|})$, we have the important result (see \cite[Lemma 3.13 \& 3.14]{Bruce:2020a})
$$M(\Lambda) \simeq \bigcup_{i \in \mathcal{I}} U_i(\Lambda) \simeq \bigcup_{x \in |M|} \Hom_{\Z_2^n\catname{Alg}}(\cO_{M,x}, \Lambda).$$
In particular, the sets $\{U_i(\Lambda) \}_{i \in \mathcal{I}}$ cover the set $M(\Lambda)$. This result allows many classical constructions to be extended to $\Z_2^n$-manifolds.  \par
\begin{remark}
$M(\Lambda)$ is not just a set, but a Fr\'echet $\Lambda_0$-manifold (here $\Lambda_0$ is the subalgebra of degree zero elements of $ \Lambda$). By restricting the natural transformations allowed, we get a full and faithful embedding of the category of $\Z_2^n$-manifolds into the category of contravariant functors from the category of $\Z_2^n$-points to the category of nuclear Fr\'echet manifolds over nuclear Fr\'echet
algebras. This embedding is the \emph{Schwarz--Voronov embedding}. For (technical) details see  \cite{Bruce:2020a}.
\end{remark}
The universal properties of the Cartesian product of two $\Z_2^n$-manifolds $M \times N$ (for details see \cite{Bruce:2018b}), directly implies that
$$\big( M \times N \big)(S) \cong M(S) \times N(S)\,.$$
This result will prove very useful when dealing with $\Z_2^n$-Lie group, their actions on $\Z_2^n$-manifolds and principal bundles.
\subsection{Fibre bundles}
As the Cartesian product of $\Z_2^n$-manifolds is well-defined as well as surjective submersions,  the notion of a fibre bundle is essentially the same as the classical one. However, one has to be a little careful as we are dealing with ringed spaces.
\begin{definition}\label{def:FibBun}
A \emph{$\Z_2^n$-fibre bundle $E=(E,M,\pi,F)$ with a typical fibre $F$} consists of a surjective submersion $\pi: E \rightarrow M$, with an atlas of \emph{local trivialisations}: an open cover $\{|U_i| \}_{i\in \mathcal{I}}$ of $|M|$ and diffeomorphisms $t_i : \pi^{-1}(U_i) \stackrel{\sim}{\rightarrow} U_i \times F$, such that the diagram
\begin{center}
\leavevmode
\begin{xy}
(0,15)*+{\pi^{-1}(U_i)}="a"; (30,15)*+{U_i \times F}="b";%
(30,0)*+{U_i}="c";%
{\ar "a";"b"}?*!/_3mm/{t_i };%
{\ar "a";"c"}?*!/^3mm/{\pi};{\ar "b";"c"}?*!/_5mm/{\textrm{prj}_1};%
\end{xy}\,,
\end{center}
where $U_i = (|U_i|,\cO_M|_{|U_i|})$, is commutative.
\end{definition}
For brevity, when all the structure is understood, we will use the nomenclature \emph{fibre bundle}. Coordinate charts on the product $M \times F$ allow us to employ local coordinates on $E$. Consider two atlases $\{(U_i, \psi_i)\}_{i \in \mathcal I}$ and $\{ (V_j, \chi_j)\}_{j \in \mathcal J}$ of $M$ and $F$, respectively. Via a refinement we can take $\{U_i \}_{\in \mathcal I}$ to be  open $\Z_2^n$-submanifolds that trivialise the  fibre bundle $E=(E,M,\pi,F)$. We then denote the  algebra map induced by the trivialisation as
$$(t^*_i)_{|U_i| \times |V_j|} : \cO_{M\times F}(|U_i|\times |V_j|) \stackrel{\sim}{\longrightarrow} \cO_E(|t_i|^{-1}(|U_i|\times |V_j|))\,.$$
From the construction of the Cartesian product (see \cite{Bruce:2018b}) we have diffeomorphisms
$$\psi_i \times \chi_j : U_i \times V_j \stackrel{\sim}{\rightarrow} \mathcal{U}_i^{p|\mathbf{q}} \times \mathcal{V}_j^{r|\mathbf{s}}\,,$$
and we write
$$\psi_i^* \,\widehat{\otimes}_\R \,\chi_j^* : C^\infty(\mathcal{U}_i^p \times \mathcal{V}_j^r)[[\zx, \eta]] \stackrel{\sim}{\rightarrow} \cO_{U_i \times V_j}(|U_i|\times |V_j|)\,,$$
using the fundamental isomorphism (see \cite[Theorem 4.2]{Bruce:2018b})
$$C^\infty(\mathcal{U}_i^p)[[\zx]] \, \widehat{\otimes}_\R \, C^\infty(\mathcal{V}_j^r)[[\eta]] \simeq C^\infty(\mathcal{U}_i^p \times \mathcal{V}_j^r)[[\zx, \eta]]\,. $$
 Then composing we can employ coordinates
$$(t^*_i)_{|U_i| \times |V_j|} \circ \psi_i^* \,\widehat{\otimes}_\R \,\chi_j^* : C^\infty(\mathcal{U}_i^p \times \mathcal{V}_j^r)[[\zx, \eta]]  \longrightarrow \cO_E(|t_i|^{-1}(|U_i|\times |V_j|))\,.$$
We then define
$$W_{ij} := (|W_{ij}|, \, \cO_E|_{|W_{ij}|})\,,$$
with $|W_{ij}| = |t_i|^{-1}(|U_i|\times |V_j|)\subseteq |E|$. We write $\phi_{ij} : W_{ij} \stackrel{\sim}{\rightarrow} \mathcal{U}_i^{p|\mathbf{q}}\times \mathcal{V}_j^{r|\mathbf{s}}$ with $\phi_{ij}= \psi_i \times \chi_j \circ t_i|_{|W_{ij}|}$.
\begin{proposition}\label{prop:FibAtlas}
Let $E=(E,M,\pi,F)$ be a fibre bundle and let $\{ (U_i, \psi_i)\}_{i \in \mathcal{I}}$ and $\{ (V_j, \chi_j)\}_{j \in \mathcal{J}}$ be atlases of $M$ and $F$, respectively. Furthermore, assume that $\{ |U_i|\}_{i \in \mathcal I}$ is a trivialising open cover of $|M|$. With the definitions above, $\{  (W_{ij}, \phi_{ij})\}_{ij \in \mathcal{I} \times \mathcal{J}}$ is an atlas on $E$ compatible with the bundle structure.
\end{proposition}
\begin{proof}
We have an open cover
 $$|E| = \bigcup_{ij \in \mathcal{I} \times \mathcal{J}} |W_{ij}| = \bigcup_{ij \in \mathcal{I} \times \mathcal{J}} |t_i|^{-1} (|U_i| \times |V_j|)\,, $$
as locally $|E|$ is diffeomorphic to the product $|U_i|\times |V_j|$. We now examine the gluing functions.\par
 Neglecting the required restrictions, we have
 $$\phi_{i'j' ij}:=\phi_{i'j'}\circ \phi_{ij}^{-1} = \psi_{i'}\times \chi_{j'} \circ (t_{i'}\circ t_i^{-1})\circ \psi_i^{-1}\times \chi_j^{-1}\,.$$
 It is clear that by construction we have the cocycle condition $\phi_{\alpha\beta}\circ \phi_{\beta \delta} \circ \phi_{\delta \alpha} = \id$, where $\alpha, \beta, \delta \in \mathcal{I}\times \mathcal{J}$ (again we neglect the restrictions). \par
 The final thing to show is that this choice of coordinates respects the bundle structure. As the trivialisation morphisms respect the projection, employing coordinates $(x^{A'}, y^{B'})$ on $\mathcal{U}_{i'}^{p|\mathbf{q}}\times \mathcal{V}_{j'}^{r|\mathfrak{s}}$ and $(x^{A}, y^{B})$ on $\mathcal{U}_{i}^{p|\mathbf{q}}\times \mathcal{V}_{j}^{r|\mathfrak{s}}$ coordinate changes are of the form (using standard abuses of notation)
 $$x^{A'} = x^{A'}(x), \quad y^{B'} =  y^{B'}(x,y)\,.$$
\end{proof}

\section{Principal bundles in the category of $\Z_2^n$-manifolds}\label{sec:PrinBun}
\subsection{$\Z_2^n$-Lie groups and actions}
We take the natural definition of a $\Z_2^n$-Lie group (see \cite{Bruce:2020a,Bruce:2020b}) following the notion of a Lie supergroup (see for example \cite[Chapter 7.1]{Varadarajan:2004}).
\begin{definition}
A \emph{$\Z_2^n$-Lie group} $G$ is a group object in the category of $\Z_2^n$-manifolds.
\end{definition}
 For the notion of a group object in a category, the reader can consult Mac~Lane \cite[page 75]{MacLane1998}.  It is clear from the definition that the reduced manifold $|G|$ is a Lie group with all the structure maps being the reduced maps.  Equivalently, via ``categorical nonsense'', we can define a $\Z_2^n$-Lie group in the following way.
\begin{definition}
A \emph{$\Z_2^n$-Lie group} is a representable group-valued functor
\begin{align*}
& G :  ~ \Z_2^n\catname{Man}^{\mathrm{op}} \longrightarrow \catname{Grp}\\
& S   \mapsto G(S) :=  \Hom_{\Z_2^n\catname{Man}}(S,G)
\end{align*}
That is, every $G(S)$ is a set-theoretical group, i.e.,
\begin{enumerate}
\item for all $\mathrm{g}, \mathrm{h}$ and $\mathrm{k} \in G(S)$, $(\mathrm{g}\cdot \mathrm{h}) \cdot \mathrm{k} =\mathrm{g}\cdot(\mathrm{h}\cdot \mathrm{k})$;
\item there exists an element $\mathrm{e} \in G(S)$ such that $\mathrm{e}\cdot \mathrm{g} =  \mathrm{g} \cdot \mathrm{e}$, for all $\mathrm{g} \in G(S)$; and
\item for every  $\mathrm{g} \in G(S)$ there exists a unique element  $\mathrm{g}^{-1} \in G(S)$, such that $\mathrm{g} \cdot \mathrm{g}^{-1} =\mathrm{g}^{-1} \cdot \mathrm{g}  = \mathrm{e}$.
\end{enumerate}
Any $\psi \in \Hom_{\Z_2^n\catname{Man}}(S' , S)$ induces a group homomorphism
\begin{align*}
G(\psi) := \psi^G :& ~G(S) \longrightarrow G(S')\\
&\mathrm{g} \mapsto \mathrm{g} \circ \psi
\end{align*}
A \emph{homomorphism of $\Z_2^n$-Lie groups} is a $\Z_2^n$-manifold morphism $\phi: G \rightarrow G'$ such that the induced map
\begin{align*}
\phi_S :& ~ G(S) \longrightarrow G'(S) \\
& \mathrm{g} \mapsto \phi \circ \mathrm{g}
\end{align*}
is a group homomorphism for all $S \in \Z_2^n\catname{Man}$.
\end{definition}
\no The resulting category of $\Z_2^n$-Lie groups we will denote as $\Z_2^n\catname{Grp}$.
\begin{remark}
We can consider generalised $\Z_2^n$-Lie groups as group-valued functors that are not necessarily representable. However, we will not encounter such functors in this paper.
\end{remark}
\begin{example}
The $\Z_2^n$-manifold $\R^{0|\mathbf{0}}$ is a trivial  $\Z_2^n$-Lie group. Specifically, $\R^{0|\mathbf{0}}$ is the terminal object in $\Z_2^n\catname{Man}$ and so there is only one element $!_S \in\R^{0|\mathbf{0}}(S)$ for any and all $S\in \Z_2^n\catname{Man}$.
\end{example}
\begin{remark}
It can easily be seen that  $\R^{0|\mathbf{0}}$ is the zero object in $\Z_2^n\catname{Grp}$.
\end{remark}
\begin{example}
Consider the functor $\textrm{Mat}(r|\textbf{s})(-) : \Z_2^n\catname{Man}^{\textrm{op}}\rightarrow \catname{Grp}$ given by
$$\textrm{Mat}(r|\textbf{s})(S):= \textrm{Mat}(r|\textbf{s}, \cO_S(|S|))\,.$$
Here, an element $X \in \textrm{Mat}(r|\textbf{s}, \cO_S(|S|))$ is a $r|\textbf{s}\times r|\textbf{s}$ block matrix with entries from $\cO_S(|S|)$. See \cite{Bruce:2020b} for more details. For any $\psi : S' \rightarrow S$ (dually $\psi^* : \cO_S(|S|) \rightarrow \cO_{S'}(|S'|)$) we have $\textrm{Mat}(r|\mathbf{s})(\psi)(\mathrm{g}):= \bar \psi^*\mathrm{g}\,,$ where $\bar \psi^*$ is $\psi^*$ acting on $\mathrm{g}$ component-wise. Then  $\textrm{Mat}(r|\textbf{s})$ is representable and so a $\Z_2^n$-Lie group under standard matrix addition.
\end{example}
\begin{example}
The general linear $\Z_2^n$-group of degree $r|\mathbf{s}$ is defined as the functor (see \cite{Bruce:2020b} for further details)
$$\textrm{GL}(r|\mathbf{s})(S) :=  \textrm{GL}(r|\mathbf{s}, \cO_S(|S|)) =  \large \{\mathrm{g}\in \textrm{Mat}(r|\mathbf{s}, \cO_S(|S|)_{\mathbf{0}} ~~|~~   \mathrm{g}~ \textnormal{is invertiable})  \large \}\,,$$
and for any $\psi : S' \rightarrow S$ (dually $\psi^* : \cO_S(|S|) \rightarrow \cO_{S'}(|S'|)$)
$$\textrm{GL}(r|\mathbf{s})(\psi)(\mathrm{g}):= \bar \psi^*\mathrm{g}\,,$$
where $\bar \psi^*$ is $\psi^*$ acting on $\mathrm{g}$  component-wise. The group product is just matrix multiplication of matrices with entries in a $\Z_2^n$-graded algebra. It was shown in \cite{Bruce:2020b} that the general linear $\Z_2^n$-group is representable.
\end{example}
\begin{example}[$\Z_2^2$-supersymmetry \cite{Bruce:2019}]\label{exp:SUSY}
Consider $\R^{1|1,1,1}$ equipped with global coordinates $(t,z, \theta^1, \theta^2)$, of degree $(0,0), (1,1), (0,1)$ and $(1,0)$, respectively.  We have a $\Z_2^2$-Lie group structure given in terms of the functor of points as
$$(t_S, z_S, \theta^1_S, \theta^2_S)\cdot(t'_S, z'_S, \theta'^{1}_S, \theta'^{2}_S):= (t_S + t'_S + \theta^1_S \theta'^1_S + \theta^2_S \theta'^2_S, z_S + z'_S +\theta^1_S \theta'^2_S - \theta^2_S \theta'^1_S, \theta^1_S + \theta'^1_S, \theta^2_S + \theta'^2_S)\,.  $$
This is, of course, no different from the ``symbolic'' way of calculating on supermanifolds as used in the physics literature (see \cite[Chapter 4.5]{Varadarajan:2004} for a discussion).
\end{example}
\begin{definition}
A \emph{right action} of a $\Z_2^n$-group $G$ on a $\Z_2^n$-manifold $M$ is a natural transformation
$$\mathrm{a}_- : M(-) \times G(-)\longrightarrow M(-) $$
that satisfies the required conditions to be an action, i.e., writing $\mathrm{a}_S(\mathrm{m} , \mathrm{g})=  \mathrm{m} \triangleleft \mathrm{g}$,
\begin{enumerate}
\item $\mathrm{m} \triangleleft \mathrm{e} = \mathrm{m}$, and \\
\item $(\mathrm{m} \triangleleft \mathrm{g})\triangleleft\mathrm{h} =  \mathrm{m}\triangleleft(\mathrm{g}\cdot \mathrm{h})$,
\end{enumerate}
for all $\mathrm{m} \in M(S)$ and $\mathrm{g}, \mathrm{h} \in G(S)$.
A right action is said to be \emph{free} if for all $S \in \Z_2^n\catname{Man}$, $\mathrm{a}_S$ is a free action, i.e., if $\mathrm{m}\triangleleft \mathrm{g} = \mathrm{m}\triangleleft\mathrm{h}$ then $\mathrm{g} = \mathrm{h}$.  A right action is said to be \emph{trivial} if for all $S \in \Z_2^n\catname{Man}$, $\mathrm{a}_S$ is a trivial action, i.e., if $\mathrm{m}\triangleleft \mathrm{g} = \mathrm{m}$ for all $\mathrm{g} \in G(S)$.
\end{definition}
\begin{example}
Let $G$ be a $\Z_2^n$-Lie group, then the right multiplication $\mathrm{a}_S(\mathrm{g}, \mathrm{h}) = \mathrm{g}\triangleleft \mathrm{h} := \mathrm{g}\cdot \mathrm{h}$, with $\mathrm{g}, \mathrm{h} \in G(S)$, is a free action.
\end{example}
\begin{example}
Let $G$ be a $\Z_2^n$-Lie group, then the  adjoint action defined as $\mathrm{g} \triangleleft \mathrm{h} := \mathrm{h}^{-1}\cdot \mathrm{g}\cdot \mathrm{h}$. One can easily check that we have the right action. Specifically
\begin{enumerate}
\item $\mathrm{g} \triangleleft \mathrm{e} =  \mathrm{e} \cdot \mathrm{g} \cdot \mathrm {e} = \mathrm{g}$,
\item $(\mathrm{g}\triangleleft \mathrm{h})\triangleleft \mathrm{k} = \mathrm{k}^{-1}\cdot \mathrm{h}^{-1} \cdot \mathrm{g} \cdot \mathrm{h}\cdot \mathrm{k} = \mathrm{g}\triangleleft \mathrm{h}\cdot \mathrm{k}$,
\end{enumerate}
with $\mathrm{g}, \mathrm{h}$ and $\mathrm{k} \in G(S)$.
\end{example}
\begin{remark}
The notion of a left action  $G \times M \rightarrow M$ is clear and will be written as $\mathrm{g}\triangleright\mathrm{m}$, using S-points.
\end{remark}
\noindent \textbf{Comment:} Note that via Yoneda's lemma a right (or left) action corresponds to a morphism of $\Z_2^n$-manifolds $\mathrm{a} : M \times G \rightarrow M$ with extra properties that are easiest to describe via the functor of points.
\begin{proposition}
Let $G$ be a $\Z_2^n$-group that acts on $\Z_2^n$-manifold $M$ from the right. Then there is a canonically induced smooth right action of $|G|$ on $|M|$
\end{proposition}
\begin{proof}
As an action corresponds to a morphism of $\Z_2^n$-manifolds it induces a morphism of smooth manifolds  $|\mathrm{a}| :  |M| \times |G| \rightarrow |M|$. We need to argue that we have an action.  Using $S =  \R^{0|\mathbf{0}} := (\star , \R)$ and making the set theoretically identifications $|M| \simeq M(\R^{0|\mathbf{0}})$ and $|G| \simeq G(\R^{0|\mathbf{0}})$, we observe that $|\mathrm{a}|= \mathrm{a}_{\R^{0|\mathbf{0}}}$. As the right action is defined in a functorial way, it is clear that $|\mathrm{a}|$  is a right action.
\end{proof}
\subsection{Principal $\Z_2^n$-bundles}
The definition of a principal bundle in the category of $\Z_2^n$-manifolds is intuitively clear following classical ideas.  The subtlety is in defining the properties of the action, which we do using the functor of points.
\begin{definition}\label{def:PrinBun}
A \emph{principal $\Z_2^n$-bundle} $P = P(M,G) = (P, M, G, \pi, \mathrm{a})$ consists of the following:
\begin{enumerate}
\item A surjective submersion $\pi : P \rightarrow M$;
\item \label{def:Prin II} An open cover $\{|U_i| \}_{i\in \mathcal{I}}$ of $|M|$ and diffeomorphisms $t_i : \pi^{-1}(U_i) \stackrel{\sim}{\rightarrow} U_i \times G$, such that the following diagram is commutative
\begin{center}
\leavevmode
\begin{xy}
(0,15)*+{\pi^{-1}(U_i)}="a"; (30,15)*+{U_i \times G}="b";%
(30,0)*+{U_i}="c";%
{\ar "a";"b"}?*!/_3mm/{t_i };%
{\ar "a";"c"}?*!/^3mm/{\pi};{\ar "b";"c"}?*!/_5mm/{\textrm{prj}_1};%
\end{xy}
\end{center}
 where $G$ is a $\Z_2^n$-Lie group, $U_i = (|U_i|,\cO_M|_{|U_i|})$ and $\pi^{-1}(U_i) = (|\pi|^{-1}|U_i|, \cO_P|_{|\pi|^{-1}|U_i|})$; and
\item A free right action  $\mathrm{a}: P \times G \rightarrow P$ that satisfies
$$t_{i, S}(\mathrm{p} \triangleleft \mathrm{g})  =  (\pi_S(\mathrm{p}), \mathrm{h}\cdot \mathrm{g}),$$
with $\mathrm{p}\in P(S)$ and $\mathrm{g}, \mathrm{h} \in G(S)$.
\end{enumerate}
\end{definition}
The fibre at an S-point $m \in U_i(S)$ is defined as
$$t_{i,S} : \pi^{-1}_S(m) \stackrel{\sim}{\rightarrow} m \times G(S) \cong G(S)\,.$$
\noindent \textbf{Comments:}  The principal action $\mathrm{a} :  P \times G \rightarrow P$
\begin{enumerate}
\item is trivial on $M$,
\item preserves the S-fibres in the sense that if $\pi_S(\mathrm{p}) =  \mathrm{m} \in M(S)$, then $\pi_S(\mathrm{p}\triangleleft \mathrm{g}) =  \mathrm{m}$ for all $\mathrm{p}\in P(S)$ and all $S \in \Z_2^n\catname{Man}$.
\end{enumerate}
\begin{definition}\label{def:TrivP}
A \emph{trivial principal $\Z_2^n$-bundle}  is a principal $\Z_2^n$-bundle of the form $ \pi : M \times G \rightarrow M $ and the action is defined via S-points as $(\mathrm{m}, \mathrm{h})\triangleleft \mathrm{g} :=  (\mathrm{m}, \mathrm{h}\cdot \mathrm{g})$, with $\mathrm{m}\in M(S)$ and $\mathrm{g}, \mathrm{h} \in G(S)$.
\end{definition}
\begin{example}
Any $\Z_2^n$-Lie group $G$ is a principal  $\Z_2^n$-bundle over a single point, with the right action being the right multiplication.
\end{example}
\begin{example}
Principal superbundles are precisely principal $\Z_2$-bundles (see \cite{Bartocci:1991, Carmeli:2018,Stavracou:1998} ).  Similarly, classical principal bundles are examples.
\end{example}
\begin{example}
Consider a principal $\mathrm{GL}(r| \mathbf{s})$-bundle $P(M, \mathrm{GL}(r| \mathbf{s}))$. Now suppose we take a trivialising open cover $\{|U_i| \}_{i \in \mathcal I}$ that is sufficiently refined such that we can employ local coordinates on $M$. That is, we have an atlas $\{ (U_i, \bar \psi_i)\}_{i \in \mathcal I}$ of $M$, where $ \bar \psi_i :  U_i \stackrel{\sim}{\rightarrow} \mathcal{U}_i^{p|\mathbf{q}} = (\mathcal{U}^p_i, C^\infty_{\mathcal{U}^p_i}[[\zx]])$, with $\mathcal{U}^p_i \subset \R^p$ is open. With these assumptions, we have diffeomorphisms
$$\bar t_i := \bar \psi_i \circ t_i : \pi^{-1}(U_i) \longrightarrow \mathcal{U}_i^p \times \mathrm{GL}(r| \mathbf{s})\,.$$
Thus, we can employ local coordinates of the form $(x^A, g_M^{\,\, N})$ (of all the required degrees). Locally we can represent S-points as pullbacks of these coordinates and we write (neglecting the restriction) $\mathrm{p}^*(x^A, g_M^{\,\, N}) = (x_S^A, (g_S)_M^{\,\, N})$, which are then collections of sections of $\cO_S(|S|)$ (using the ``chart theorem'', see \cite{Covolo:2016}). The action is then, locally, of the form
$$(x_S^A, (g_S)_M^{\,\,\,  N})\triangleleft (h_S)_N^{\,\, L} =  (x_S^A, (g_S)_M^{\,\,\,  N}\, (h_S)_N^{\,\, L})\,.$$
The action, locally,  is no more than standard matrix multiplication of matrices with entries in $\cO_S(|S|)$.
\end{example}
\begin{remark}
The previous example is somewhat generic as most, but not all, of the interesting Lie groups and Lie supergroups found in geometry are matrix Lie (super)groups. While $\Z_2^n$-Lie groups are an almost completely unexplored topic, we expect the matrix case to be prominent.
\end{remark}
\begin{proposition}\label{prop:RedPrinBun}
Let $P(M,G)$ be a principal $\Z_2^n$-bundle.  The reduced manifold $|P|$ is a principal $|G|$-bundle.
\end{proposition}
\begin{proof}
From the definitions, we know that $|P|$ is a smooth manifold and that underlying the surjective submersion of the $\Z_2^n$-manifolds we have a surjective submersion of smooth manifolds $|\pi|: |P| \rightarrow |M|$ (see Lemma \ref{lem:SurSub}). Moreover,  the trivialisations have underlying them trivialisations of the reduced manifolds, i.e., we have a family of diffeomorphisms $|t_i| : \, |\pi|^{-1}(|U_i|) \stackrel{\sim}{\rightarrow} |U_i| \times |G|$, here $\{ |U_i|\}_{i \in \mathcal I}$ is an open cover of $|M|$ that trivialises $P$. These maps satisfy the required diagrams to define a fibre bundle in the category of smooth manifolds with typical fibre being the Lie group $|G|$.\par
The only thing remaining is to show that the principal action on $P$ induces a principal action on $|P|$. As the action is defined as a morphism of $\Z_2^n$-manifolds (with extra properties) we know that there is an underlying smooth map $|\mathrm{a}| : |P| \times |G| \rightarrow |P|$. Using $S =  \R^{0|\mathbf{0}} := (\star , \R)$ and making the set theoretically identifications $|P| \simeq  P(\R^{0|\mathbf{0}})$ etc., we make the further identification $|\mathrm{a}|= \mathrm{a}_{\R^{0|\mathbf{0}}}$. As the principal action is defined in a functorial way, it is clear that $|\mathrm{a}|$  is a free right action. Similarly, the functorial nature of the constructions show that $t_{i,\R^{0|\mathbf{0}}}(p\triangleleft g) = (|\pi|(p), h \cdot g)$, with $p \in |P|$, $h,g\in |G|$.  Thus, we have an underlying principal $|G|$-bundle.
\end{proof}
Coordinate charts on the product $M \times G$ allow us to employ local coordinates on $P(M,G)$. As a specific case of Proposition \ref{prop:FibAtlas} we have the following.
\begin{proposition}\label{prop:Atlas}
Let $P(M,G)$ be a principal $G$-bundle and let $\{ (U_i, \psi_i)\}_{i \in \mathcal{I}}$ and $\{ (V_j, \chi_j)\}_{j \in \mathcal{J}}$ be atlases of $M$ and $G$, respectively. Furthermore, assume that $\{ |U_i|\}_{i \in \mathcal I}$ is a trivialising open cover of $|M|$. With the definitions above, $\{  (W_{ij}, \phi_{ij})\}_{ij \in \mathcal{I} \times \mathcal{J}}$ is an atlas on $P$ compatible with the bundle structure.
\end{proposition}
Let $P(M,G)$ be a principal $\Z_2^n$-bundle. We have a relation $\mathrm{p} \sim \mathrm{p}'$ on $P(S)$ given by $\exists ~ \mathrm{g} \in G(S)$ such that $\mathrm{p} = \mathrm{p}'\triangleleft \mathrm{g}$. It is easy to see that we have a relation in this way.
\begin{lemma}
Let $\psi : S' \rightarrow S$ be a morphism of $\Z_2^n$-manifolds. If $\mathrm{p}$ and $\mathrm{p}' \in P(S)$ are related, then $\psi^P(\mathrm{p})$ and $\psi^P(\mathrm{p}') \in P(S')$ are related.
\end{lemma}
\begin{proof}
This follows from the functorial property of the definition of the actions, i.e.,
 $$\psi^P(\mathrm{p}) =  \psi^P(\mathrm{p}'\triangleleft\mathrm{g}) = \psi^P(\mathrm{p}')\triangleleft \psi^P(\mathrm{g})\,.$$
\end{proof}
\begin{definition}
The \emph{orbit space} of $P(M,G)$ is defined as the functor
\begin{align*}
& P\slash G : ~ \Z_2^n\catname{Man}^{\textrm{op}} \longrightarrow \catname{Set}\\
& S \longmapsto (P\slash G)(S) := P(S)\slash \sim
\end{align*}
and for  morphism $\psi :  S' \rightarrow S$ we set $\psi^{P\slash G}(\mathrm{p}) =  \mathrm{p} \circ \psi$.
\end{definition}
Consider the trivial principal $\Z_2^n$-bundle $M \times G$. Then, for any pair $(\mathrm{m}, \mathrm{h}) \in M(S) \times G(S)$ the action is $(\mathrm{m}, \mathrm{h})\triangleleft \mathrm{g} =  (\mathrm{m}, \mathrm{h}\cdot \mathrm{g})$. Thus, the equivalence relation is only on $G(S)$ and  so
$$\big((M \times G) \slash G  \big)(S) = M(S) \times \mathrm{e} \cong M(S)\,,$$
as every pair $(\mathrm{m}, \mathrm{h}) \sim (\mathrm{m} , \mathrm{e})$, by just setting $\mathrm{g} =  \mathrm{h}^{-1}$. Thus, as expected, the orbit space of a trivial $\Z_2^n$-principle bundle is representable and is naturally isomorphic to $M$ under the image of the Yoneda embedding. This local result implies the global result which we state as the following proposition.
\begin{proposition}\label{prop:OrbitSpace}
Let $P(M,G)$ be a principle $\Z_2^n$-bundle. Then the associated orbit space $P\slash G$ is representable and  `the' representing $\Z_2^n$-manifold is $M$.
\end{proposition}
\begin{proof}
It will be convenient to use $\Lambda$-points. First note that $\{|\pi|^{-1}(|U_i|) \}_{i \in \mathcal{I}}$ is an open cover of $|P|$. Then, using the local trivialisations and Lemma 3.13 and Lemma 3.14 of \cite{Bruce:2020a} we observe that
$$P(\Lambda) \simeq \bigcup_{i \in \mathcal I} ~ \pi^{-1}(U_i)(\Lambda) \simeq \bigcup_{i \in \mathcal I} ~\big( U_i(\Lambda)\times G(\Lambda)\big)\,.$$
Then the equivalence relation is on the second factor of pairs and so
$$(P\slash G)(\Lambda) \simeq \bigcup_{i \in \mathcal I} ~ U_i(\Lambda) \simeq M(\Lambda)\,.$$
Thus, via the Yoneda lemma, we have $(P\slash G) \simeq M$ as $\Z_2^n$-manifolds.
\end{proof}
\noindent \textbf{Aside.} One can ask if the definition of a principal $\Z_2^n$-bundle can be defined a little more generally as a fibre bundle with an arbitrary fibre that has a free and transitive action of $G$ that acts trivially on the base $\Z_2^n$-manifold. To address this question, it is sufficient to examine the trivial bundle $P = M \times F$.  Consider the natural transformation $\Phi_-$ defined component-wise as
\begin{align*}
\Phi_S ~ : M(S)\times F(S)\times G(S)& \longrightarrow M(S)\times F(S)\times F(S)\\
 (m, f, g)& \longmapsto (m, f , f \triangleleft g )
\end{align*}
Observe that for any $S \in \Z_2^n\catname{Man}$, the map $\Phi_S$ is a bijection: injectivity follows from the fact that the map acts as the identity on the first two factors and that the action is a free action, and surjectivity follows as the action is a transitive action. Thus, $G(S) = F(S)$ as sets. Then using the Yoneda lemma we conclude that the $\Z_2^n$-manifolds $G$ and $F$ are diffeomorphic. This is, of course, analogous to the classical case.

\subsection{Morphisms of principal $\Z_2^n$-bundles}
The notion of a morphism of principal $\Z_2^n$-bundles follows from the concept of a morphism of a classical principal bundle with the appropriate modifications.
\begin{definition}
A \emph{morphism of principal $\Z_2^n$-bundles} $P(M,G) \rightarrow P'(M', G')$ is a triple $(\Phi,\phi, \psi)$, where both $\Phi : P \rightarrow P'$ and $\phi : M \rightarrow M'$ are morphisms of $\Z_2^n$-manifolds, and $\psi:  G \rightarrow G'$ is a homomorphism of $\Z_2^n$-Lie groups, such that:
\begin{enumerate}
\item The pair $(\Phi, \phi)$ is a bundle map, i.e., the following diagram is commutative
\begin{center}
\leavevmode
\begin{xy}
(0,15)*+{P}="a"; (20,15)*+{P'}="b";%
(0,0)*+{M}="c"; (20,0)*+{M^\prime}="d";%
{\ar "a";"b"}?*!/_3mm/{\Phi };%
{\ar "a";"c"}?*!/^3mm/{\pi};{\ar "b";"d"}?*!/_3mm/{\pi^\prime};%
{\ar "c";"d"} ?*!/^3mm/{\phi};%
\end{xy}
\end{center}
In terms of S-points, we have $\pi'_S \circ \Phi_S =  \phi_S \circ \pi_S$ for all $S\in \Z_2^n\catname{Man}$.
\item The morphism $\Phi$ is compatible with the action in the sense that $\Phi_S(\mathrm{p}\triangleleft \mathrm{g}) = \Phi_S(\mathrm{p})\triangleleft \psi(\mathrm{g})$.
\end{enumerate}
\emph{Isomorphisms of principal $\Z_2^n$-bundles} are triples $(\Phi, \phi, \psi)$ consisting on isomorphisms in the relevant categories.
\end{definition}
As a matter of formality, we have several categories of interest here:
\begin{center}
\renewcommand{\arraystretch}{1.5}
\begin{tabular}{l l}
$\Z_2^n\catname{Prin}$ & The category of principal $\Z_2^n$-bundles \\
$\Z_2^n\catname{Prin}(M)$ & The category of principal $\Z_2^n$-bundles over $M$, with $\phi = \Id_M$\\
$\Z_2^n\catname{Prin}(M,G)$ & The category of principal $G$-bundles over $M$, with $\phi = \Id_M$ and $\psi =  \Id_G$.
\end{tabular}
\end{center}
\begin{proposition}\label{prop:AllInvTriv}
All endomorphisms of the trivial principal $G$-bundle $M\times G  \in \Z_2^n\catname{Prin}(M,G)$ (so $\phi = \Id_M$ and $\psi =  \Id_G$) are automorphisms, i.e., they are always invertible.
\end{proposition}
\begin{proof}
Let $\Phi$ be an endomorphism of a trivial principal bundle $M \times G$ (from the definition the other maps are identity maps).   Then for any $S\in \Z_2^n\catname{Man}$, $\Phi_S(\mathrm{m}, \mathrm{h}) = (\mathrm{m}, \lambda_S(\mathrm{m})\cdot\mathrm{h})$, where $\lambda_- : M(-) \rightarrow G(-)$ is a natural transformation between the respective sets of S-points.  Using Yoneda's lemma this corresponds to a $\Z_2^n$-manifold morphism $\lambda : M \rightarrow G$. Note that $\Phi_S$ is a bijective function from $M(S)\times G(S)$ to itself as we can define
$$\Phi^{-1}_S(\mathrm{m}, \mathrm{h}) = (\mathrm{m}, (\lambda(\mathrm{m}) )^{-1}\cdot \mathrm{h} ).$$
Then via Yoneda's lemma, we observe that $\Phi^{-1}$ exists as a morphism of $\Z_2^n$-manifolds.
\end{proof}
The previous results on trivial principal G-bundles imply the non-trivial result.
\begin{proposition}\label{prop:AllIso}
All morphisms in the category $\Z_2^n\catname{Prin}(M,G)$ are isomorphisms.
\end{proposition}
\begin{proof}
Let $\Phi : P \rightarrow P'$ be a morphism of principal G-bundles over the same base. Note that using $\Lambda$-points $P(\Lambda) \simeq P'(\Lambda) \simeq \bigcup_{i \in \mathcal I} \big (U_i(\Lambda) \times G(\Lambda) \big)$, where we have chosen open cover $\{ |U|_i\}_{i \in \mathcal I}$ of $|M|$ that trivialises both  $P$ and $P'$. Such a cover can always be found by refining the covers given in the definition of a principal $\Z_2^n$-bundle. We then consider the following commutative diagram.
\begin{center}
\leavevmode
\begin{xy}
(0,20)*+{\pi_\Lambda^{-1}(U_i(\Lambda))}="a"; (45,20)*+{\pi'^{-1}_\Lambda(U_i(\Lambda))}="b";%
(0,0)*+{U_i(\Lambda)\times G(\Lambda)}="c"; (45,0)*+{U_i(\Lambda)\times G(\Lambda)}="d";%
{\ar "a";"b"}?*!/_3mm/{\Phi_\Lambda |_{\pi_\Lambda^{-1}(U_i(\Lambda))} };%
{\ar "a";"c"}?*!/^3mm/{t_{i, \Lambda}};{\ar "b";"d"}?*!/_3mm/{t'_{i, \Lambda}};%
{\ar "c";"d"} ?*!/^3mm/{\Psi_{i, \Lambda}};%
\end{xy}
\end{center}
Thus, $\Psi_{i, \Lambda} := t'_{i, \Lambda} \circ \Phi_\Lambda |_{\pi_\Lambda^{-1}(U_i(\Lambda))} \circ t_{i, \Lambda}^{-1}$, which is invertible via Proposition \ref{prop:AllInvTriv}. Hence,  $ \Phi_\Lambda |_{\pi_\Lambda^{-1}(U_i(\Lambda))}$ is itself invertible. As $P(\Lambda) \simeq \bigcup_{i\in \mathcal{I}} \pi_\Lambda^{-1}(U_i(\Lambda))$ we deduce that $\Phi_\Lambda$ is invertible. Then, via Yoneda's lemma, we conclude that $\Phi$ is a diffeomorphism in the category of $\Z_2^n$-manifolds. Thus, the proposition is established.
 \end{proof}
\begin{remark}
In complete analogy with the classical case, $\Z_2^n\catname{Prin}(M,G)$ ($n\geq 1$) is a groupoid, but now not small, i.e., the objects are not sets.
\end{remark}
\begin{remark}
For a given principal $G$-bundle, the automorphisms in $\Z_2^n\catname{Prin}(M,G)$ form what is known as the \emph{gauge group} of $P$, and is denoted $\textrm{Gau}(P)$. Note that this is a group and not a $\Z_2^n$-graded group as all the maps carry zero $\Z_2^n$-degree. To describe the ``full gauge group'', one needs to enrich $\textrm{Gau}(P)$ using the internal homs, in this way, we get gauge transformations that carry non-zero  $\Z_2^n$-degree. We will not spell out the construction in this paper.
\end{remark}
We will now turn our attention to the characterisation of trivial principal $G$-bundles in terms of the existence of a global section. By a  \emph{section} of a principal $G$-bundle $\zp:P\to M$, we mean a morphism of $\Z_2^n$-manifolds $s: M \rightarrow P$ such that $\pi \circ s =  \Id_M$. That is, we only consider degree zero sections and do not consider enriched sections in this paper, i.e., sections that carry non-zero $\Z_2^n$-degree.
Using $\Lambda$-points we observe that a section corresponds to a map
$$s_\Lambda : \bigcup_{i \in \mathcal I} U_i(\Lambda) \longrightarrow \bigcup_{i \in \mathcal I} \big (U_i(\Lambda)\times G(\Lambda)\big)\,, $$
given by $U_{i}(\Lambda) \ni m \mapsto (m, g)\in U_i(\Lambda)\times G(\Lambda)$. That is, to every $\Lambda$-point of $M$ we assign an element of the group $G(\Lambda)$.
\begin{proposition} \label{prop:GlobSecTriv}
A principal $G$-bundle $P(M,G)$ is trivialisable if and only if it admits a global section.
\end{proposition}
\begin{proof}
Consider the trivial  principal $G$-bundle $\pi: M \times G \rightarrow M$  The specification of any morphism of $\Z_2^n$-manifolds $\lambda : M \rightarrow G$ defines a section via $\mathrm{m} \mapsto (\mathrm{m}, \lambda_S(\mathrm{m}))$, where $\lambda_S(\mathrm{m}) := \lambda \circ \mathrm{m}$, for  any $S \in \Z_2^n\catname{Man}$.  Thus, we can construct global sections of trivial principal $G$-bundles. \par
Suppose that we have a section $s : M \rightarrow P$. We then define a morphism of $\Z_2^n$-manifolds $\Phi : M \times G \rightarrow P$ using S-points as $\Phi_S(\mathrm{m}, \mathrm{g}) := s_S(\mathrm{m}) \triangleleft \mathrm{g}$, where $s_S(\mathrm{m}):= s \circ \mathrm{m}$. We now need to argue that this morphism is a morphism of principal $\Z_2^n$-bundles.  First, it is clear that we have a bundle morphism as
 $$\pi_S(\Phi_S(\mathrm{m}, \mathrm{g})) = \pi_S(s_S(\mathrm{m})\triangleleft\mathrm{g}) = \mathrm{m},$$
as required. Similarly, a short calculation shows that this morphism is compatible with the action,
$$\Phi_S((\mathrm{m} , \mathrm{g})\triangleleft\mathrm{h}) =  \Phi_S(\mathrm{m}, \mathrm{g}\cdot \mathrm{h}) = s_S(\mathrm{m}) \triangleleft\mathrm{g}\cdot \mathrm{h}= (s_S(\mathrm{m}) \triangleleft \mathrm{g})\triangleleft \mathrm{h} =  \Phi_S(\mathrm{m}, \mathrm{g})\triangleleft\mathrm{h}\,$$
as required. Then using Proposition \ref{prop:AllIso}, we see that $\Phi$ is an isomorphism and thus we can conclude.
\end{proof}
\begin{proposition}
The category $\Z_2^n\catname{Prin}(M)$ admits finite products.
\end{proposition}
\begin{proof}
Following the classical construction, we define the product via S-points (so as a functor from $\Z_2^n\catname{Man}^{\textrm{op}} \rightarrow \catname{Set}$) as
$$\big(P\times_M P' \big)(S) := \left \{(p,p') \in P(S)\times P'(S) ~~| ~~\pi_S(p) = \pi'_S(p') \right\}\,.$$
This is just the product of $\pi_S : P(S) \rightarrow M(S)$ and  $\pi'_S : P'(S) \rightarrow M(S)$ in the category of sets. Thus, via Yoneda's lemma, provided $P\times_M P'$ exists in $\Z_2^n\catname{Man}$, we have the required universal properties of a product.\par
The representability of $\big(P\times_M P' \big)(-)$ is easiest to observe using $\Lambda$-points. Specifically, by using an open cover of $M$ that trivialises both $P$ and $P'$ we see that
 $$\big(P\times_M P' \big)(\Lambda) \simeq \bigcup_{i \in \mathcal{I}}\big(  U_ i(\Lambda) \times G(\Lambda) \times G'(\Lambda)\big)\,. $$
 Then via \cite[Theorem 3.34]{Bruce:2020a}, we see that the product has the local structure required to be representable, the functor has a cover built from open subfunctors. Thus, $P\times_M P'$ is a $\Z_2^n$-manifold. \par
 The fibre bundle structure is evident and the action of $G$ and $G'$, as defined by S-points, is component-wise in the product.
\end{proof}
\subsection{Transition functions and gluing}
Given an open in $\{ |U_i|\}_{i \in \mathcal I}$ that trivialises a principal $\Z_2^n$-bundle $P(M,G)$, we have a  diffeomorphisms (in the category of $\Z_2^n$-manifolds)
$$t_i : \pi^{-1}(U_i) \stackrel{\sim}{\rightarrow} U_i \times G\,.$$
We then define, as standard, $|U_{ij}| = |U_i|\cap |U_j|$ and  $|U_{ijk}| = |U_i|\cap |U_j| \cap |U_k|$, taken to be non-empty.  We define $U_{ij} := (|U_{ij}|, \cO_M|_{|U_{ij}|})$.  As standard, we have the diffeomorphisms
\begin{equation}\label{eqn:tij}
t_{ij} := t_i|_{|U_{ij}|} ~\circ ~ \big( t_j|_{|U_{ij}|}\big)^{-1} :  U_{ij}\times G \longrightarrow U_{ij} \times G\,.
\end{equation}
We will examine these diffeomorphisms using S-points. Suppose $\mathrm{p} \in \pi_S^{-1}(U_{ij}(S))$ is such that $t_{i,S}(\mathrm{p}) = (\mathrm{x} , \mathrm{h})$ and $t_{j,S}(\mathrm{p}) = (\mathrm{x} , \mathrm{g})$. We write, $t_{ij,S}(\mathrm{x}, \mathrm{g}) = (\mathrm{x} , \mathrm{h}) = (\mathrm{x} , \psi_{ij, S}(\mathrm{x})\cdot\mathrm{g})$, with $\psi_{ij, S}: U_{ij}(S) \rightarrow G(S)$.  The maps $t_{ij,S}$ are invertible as they are constructed from diffeomorphisms. Then, using Yoneda's lemma, we naturally arrive at the following definition.
 \begin{definition}\label{def:TransMorp}
The $\Z_2^n$-manifold morphisms  $\psi_{ij}: U_{ij} \rightarrow G$ constructed from the isomorphisms $t_{ij}$ are called the \emph{transition morphisms} for $P(M,G)$ relative to the chosen trivialisation $\big\{ (|U_i|, t_i) \big\}_{i \in \mathcal I}$.
\end{definition}
 Via the definitions it is clear that $\psi_{ij,S}(\mathrm{x}) =  \psi_{ij}\circ \mathrm{x}$. Then under change of parametrisation $\psi : S' \rightarrow S$ we have $\psi_{ij,S}(\mathrm{x}) \mapsto \psi_{ij}\circ \mathrm{x}\circ \psi =  \psi_{ij, S'}(\mathrm{x}')$ with $\mathrm{x}' := \mathrm{x}\circ \psi$ for any $\mathrm{x} \in U_{ij}(S)$.
 \begin{proposition}\label{prop:TransMaps}
The diffeomorphisms defined in \eqref{eqn:tij}, are isomorphisms of trivial principal $\Z_2^n$-bundles and satisfy the gluing conditions:
\begin{enumerate}
\item $t_{ii} = \Id$,
\item $t_{ij} \circ t_{ji} = \Id$, and
\item $t_{ik}|_{|U_{ijk}|}\circ t_{kl}|_{|U_{ijk}|} \circ t_{ki}|_{|U_{ijk}|}= \Id$,
\end{enumerate}
where it is understood where the identity maps live.
 \end{proposition}
\begin{proof}
Suppose $\mathrm{p} \in \pi_S^{-1}(U_{ij}(S))$ is such that  $t_{ij,S}(\mathrm{x}, \mathrm{g}) = (\mathrm{x} , \mathrm{h}) = (\mathrm{x} , \psi_{ij, S}(\mathrm{x})\cdot\mathrm{g})$. It is clear that the diffeomorphisms $t_{ij}$  are bundle morphisms (using Yoneda's Lemma).  A quick calculation shows that $t_{ij,S}((\mathrm{x},\mathrm{g})\triangleleft \mathrm{k}) = t_{ij,S}(\mathrm{x},\mathrm{g}\cdot \mathrm{k}) = (\mathrm{x},\psi_{ij, S}(\mathrm{x})\cdot \mathrm{g} \cdot \mathrm{k} ) = t_{ij,S}(\mathrm{x},\mathrm{g})\triangleleft \mathrm{k}$. Thus, we have an isomorphism of trivial principal $\Z_2^n$-bundles. The gluing conditions are obvious via construction.
\end{proof}
\begin{lemma}\label{lem:Cocycle}
From the gluing conditions we have $\psi_{ik,S}(\mathrm{x})\cdot \psi_{kl,S}(\mathrm{x})\cdot \psi_{li,S}(\mathrm{x})=\Id$ for any S-point $\mathrm{x} \in U_{ijk}(S)$ (neglecting the required restrictions). Equivalent, via Yoneda's lemma, the transition morphisms satisfy
\begin{equation}\label{eqn:Cocycle}
 \psi_{ik,}\cdot \psi_{kl}\cdot \psi_{li}=\Id\,.
 \end{equation}
\end{lemma}
In the other direction, given an appropriate collection of morphisms that satisfy the cocycle condition, we can always find a principle $\Z_2^n$-bundle whose transition functions are given by the said collection of morphisms. In other words, specifying, for a given $M$, a collection $\{(|U_i|, \psi_{ij}) \}$ such that the morphisms satisfy the cocycle, defines, up to isomorphism, a unique principle $\Z_2^n$-bundle.
\begin{theorem}\label{thm:CocyGlue} Let  $M = (|M|, \cO_M)$ be a $\Z_2^n$-manifold together with a fixed open covering $\{|U_i| \}_{i \in \mathcal I}$ and a collection of maps  $\psi_{ij}: U_{ij} \rightarrow G$ that satisfy the cocyle condition \eqref{eqn:Cocycle}. Then there exist a $G$-principal bundle $P(M,G)$ and a trivialistion of it over the open cover $\{|U_i|\}_{i \in \mathcal I}$ of $|M|$ whose transition morphisms are $\psi_{ij}$.
\end{theorem}
  \begin{proof} Starting form the open cover $\{ |U_i|\}_{i \in \mathcal I}$ and $G$  we want to glue `spaces' $U_i \times G$ to construct a $\Z_2^n$-manifold.  As we have maps of $\Z_2^n$-manifolds $\psi_{ij}= (|\psi_{ij}|, \psi^*_{ij}) : U_{ij} \rightarrow G$ we have underlying maps of smooth manifolds $|\psi_{ij}|: |U_{ij}| \rightarrow |G|$ that satisfy the cocycle condition. In particular, we have diffeomorphims $t_{ij} : U_{ij}\times G \rightarrow U_{ij} \times G$.  We can then use the classical theorem to construct a fibre bundle in the category of smooth manifolds $|\pi| :  |P| \rightarrow |M|$ with local trivialisation $|\pi|^{-1}(|U_i|) \simeq |U_i| \times |G|$. We now need to construct the structure sheaf $\cO_P :  |V| \rightarrow \cO_P(|V|)$, with $|V|\subset |P|$ is open, in order to have a $\Z_2^n$-manifold.  We define $P_i := U_i \times G$  and $|P_{ij}| :=  |P_i| \cap |P_j| \cong |U_{ij}|\times G$. Then we set  for any open $|V| \subseteq |P|$
$$\cO_P(|V|) :=  \left \{\big(s_i \in \cO_{P_i}(|P_i|\cap |V|) \big)_{i \in \mathcal{I}} ~~|~~ t^*_{ij}(s_i|_{|P_{ij}|\cap |V|}) = s_j|_{|P_{ij}|\cap |V|}\, , ~~ i,j \in \mathcal{I}  \right\}\,.$$
Observe that the cocycle condition implies that $\cO_P|_{|\pi|^{-1}(|U_i|)} \simeq \cO_{P_i} \simeq \cO_{U_i \times G}$. We thus have a fibre bundle in the category of $\Z_2^n$-manifolds $P=(|P|, \cO_P)$. Note that by construction
$$\pi^{-1}(U_i) \simeq (|U_i| \times |G|, \cO_{U_i \times G})\,.$$
We now need to construct the principal right action.  Using $\Lambda$-points we have
$$P(\Lambda) \simeq \bigcup_{i \in \mathcal I} \big( U_i(\Lambda) \times G(\Lambda) \big)\,.$$
 We then define the action as $(\mathrm{x}, \mathrm{g})\triangleleft \mathrm{h} = (\mathrm{x}, \mathrm{g}\cdot \mathrm{h})$. This action satisfies the conditions we need.
  \end{proof}
\subsection{Associated bundles}
Consider a $P(M,G)$ be a principal $\Z_2^n$-bundle and a $\Z_2^n$-manifold $F$ equipped with a smooth left action  $\mathrm{b}: G \times F \rightarrow F$.  We take this action of be faithful, i.e., the action $\mathrm{b}_S :  G(S)\times F(S) \rightarrow F(S)$ is faithful for all $S  \in \Z_2^n\catname{Man}$.  That is, for every $\mathrm{g} \neq \mathrm{e} \in G(S)$ there exists $\mathrm{v} \in F(S)$ such that $\mathrm{g}\triangleright \mathrm{v} \neq \mathrm{v}$. We chose a trivialising open cover $\{ |U_i|\}_{i \in \mathcal{I}}$ and transition functions $t_{ij} : U_{ij}\times G \rightarrow U_{ij} \times G$, which are given via S-points as $t_{ij,S}(\mathrm{x} , \psi_{ij,S}(\mathrm{x})\cdot \mathrm{g})$, with $\psi_{ij}(\mathrm{x}):= \psi_{ij}\cdot \mathrm{x}$.  We construct a fibre bundle $E$ over $M$
with fibre $F$ via specifying the transition maps $\tau_{ij}= (|\tau_{ij}|, \tau^*_{ij}): U_{ij} \times F \rightarrow U_{ij} \times F$ defined via S-points as
$$\tau_{ij,S}(\mathrm{x}, \mathrm{v}) :=  (\mathrm{x},  \psi_{ij,S}(\mathrm{x}) \triangleright \mathrm{v})\,.$$
By constriction, these transition maps satisfy the cocycle condition. We then use the classical theorem to construct a fibre bundle in the category of smooth manifolds $|\pi| :  |E| \rightarrow |M|$ with local trivialisation $|\pi|^{-1}(|U_i|) \simeq |U_i| \times |F|$. We now need to construct the structure sheaf $\cO_E :  |V| \rightarrow \cO_E(|V|)$, with $|V|\subset |E|$ is open, in order to have a $\Z_2^n$-manifold.  We define $E_i := U_i \times F$  and $|E_{ij}| :=  |E_i| \cap |E_j| \simeq |U_{ij}|\times F$. Then we set  for any open $|V| \subseteq |E|$
$$\cO_E(|V|) :=  \left \{\big(s_i \in \cO_{E_i}(|E_i|\cap |V|) \big)_{i \in \mathcal{I}} ~~|~~ \tau^*_{ij}(s_i|_{|E_{ij}|\cap |V|}) = s_j|_{|E_{ij}|\cap |V|}\, , ~~ i,j \in \mathcal{I}  \right\}\,.$$
Observe that the cocycle condition implies that $\cO_E|_{|\pi|^{-1}(|U_i|)} \simeq \cO_{E_i}$. We thus have a fibre bundle in the category of $\Z_2^n$-manifolds $E=(|E|, \cO_E)$.
\begin{definition}
Let $P(M,G)$ be a principal $\Z_2^n$-bundle and let $F$ be a $\Z_2^n$-manifold equipped with a faithful smooth left action  $\mathrm{b}: G \times F \rightarrow F$.  Then the $\Z_2^n$-fibre bundle $E$, constructed above, is the \emph{associated fibre $\Z_2^n$-bundle} to $P(M,G)$ with fibre $F$.
\end{definition}

\section{Vector bundles and frame bundles}\label{sec:Fbundles}

\subsection{Vector bundles in the category of $\Z_2^n$-manifolds}
We cannot think of a vector bundle in the setting of $\Z_2^n$-geometry as a $\Z_2^n$-manifold in which we attach a graded vector space to each point. In particular, the fibres cannot be real vector spaces, but linear or Cartesian  $\Z_2^n$-manifolds.  The vector space structure is somewhat secondary in this picture.  Cartesian $\Z_2^n$-manifold were first properly discussed in \cite{Bruce:2020b}, and so direct the reader there for further details and results.
\begin{definition}
A \emph{Cartesian $\Z_2^n$-manifold} of dimension $p|\mathbf{q}$ is a $\Z_2^n$-manifold of the form
$$\R^{p|\mathbf{q}} = \big(\R^p, C^\infty_{\R^p}[[\zx]] \big)\,,$$
where $\mathbf{q} = (q_1, q_2, \cdots , q_{2^n-1})$ and $\zx^\za$ are formal coordinates of all non-zero $\Z_2^n$-degrees such that we have $q_i$ formal variables of $\Z_2^n$-degree $\gamma_i$ ($i >0$). We employ canonical linear coordinates $x^a$ on $\R^p$ and formal coordinates $\zx^\za$, and collectively write $x^A = (x^a, \zx^\za)$.
\end{definition}
\begin{remark}
We also have $\Z_2^n$-graded vector spaces $\mathbf{R}^{p|\mathbf{q}} := \mathbf{R}^p \displaystyle \bigotimes_{j=1}^N \mathbf{R}^{q_j}$  (here each $\mathbf{R}^\bullet := \R^\bullet$ as vector spaces). Note we have the important result
$$\big(\R^{p|\mathbf{q}}  \big)^\vee  \simeq \cO^{\textnormal{lin}}_{ \R^{p|\mathbf{q}}}(\R^p)\,.$$
\end{remark}
\no A very convenient tool to work with vector bundles is the concept of a \emph{homogeneity structure} on a $\z2$-manifold $M$, understood as $\z2$-degree $\mathbf{0}$ vector field $\n$, of degree 0 (cf. \cite{Grabowski:2009,Grabowski:2012}). The pair $(M,\n)$ we call a \emph{homogeneity manifold}. Then, we call a (local) function $f$ on $M$ \emph{$\n$-homogeneous of weight $w\in\R$}, if $\n(f)=w\cdot f$. Obviously, the product $ff'$ of a $w$-homogeneous function $f$ and a $w'$-homogeneous function $f'$ is $(w+w')$-homogeneous. This can be extended to the homogeneity of arbitrary tensor fields with the use of a Lie derivative. A morphism between homogeneity manifolds $(M,\n)$ and $(M',\n')$ is a smooth map $\zf:M\to M'$ which relates the weight vector fields. In other words, the pullbacks of (local) $w$-homogeneous functions are (local) $w$-homogeneous functions.\par
\mn Consider now a Cartesian manifold  $M=\R^{p|\bq}$ with global $\z2$-homogeneous coordinates $(x^a,\zx^\za)$. The obvious linear structure on the $\R$-vector space $\mathbf{R}^{p|\bq}$, dual to the one spanned by these coordinates, is encoded in the \emph{Euler vector field} on $\R^{p|\bq}$,
\be\label{euler}
\n^{p|\bq}=\sum_ax^a\pa_{x^a}+\sum_\za\zx^\za\pa_{\zx^\za},
\ee
which, clearly, has degree 0. Note that we can also view $\R^{p|\bq}$ as a $\z2$-graded $\R$-vector space spanned by $\pa_{x^a},\pa_{\zx^\za}$.

A useful observation is that the degree 0 vector field $\n^{p|\bq}$ is exactly the generator of the multiplication by positive reals in $M=\R^{p|\bq}$: if
$$h_t(x^a,\zx^\za)=(e^tx^a,e^t\zx^\za)$$
for $t\in\R$, then $h_t$ is a one-parameter group of diffeomorphisms of $\R^{p|\bq}$ and
$$\n^{p|\bq}(x^a,\zx^\za)=\frac{\xd}{\xd t}_{\big|_{t=0}}h_t(x^a,\zx^\za).$$
\begin{corollary}
The Euler vector field $\n^{p|\bq}$ is completely determined by the multiplication by positive reals in $M=\R^{p|\bq}$, so it does not depend on the choice of linear coordinates and therefore is preserved by linear diffeomorphisms of the $\z2$-manifold $\R^{p|\bq}$.
\end{corollary}
\begin{proposition}[\textbf{Euler's Homogeneous Theorem}] Let $f$ be a $\n^{p|\bq}$-homogeneous of weight 1 function on the $\z2$-manifold $\R^{p|\bq}$. Then,
$$f=\sum_a F_a{x^a}+\sum_\za G_\za{\zx^\za},$$
for some $F_a,G_\za\in\R$. In other words, $f$ is a linear function in the standard sense.
\end{proposition}
\begin{proof} The proof easily follows from the Euler's Homogeneous Theorem on $\R^p$, and the observation that
$$\zx^i\pa_{\zx^i}(\zx^\bza)=\bza_i\zx^\bza.$$
\end{proof}
\no In what follows, we will view \emph{linear functions} on the manifold $\R^{p|\bq}$ as smooth functions which are 1-homogeneous with respect to $\n^{p|\bq}$.

\mn By \emph{trivial $\z2$-vector bundles of rank ${p|\bq}$} we will understand the Cartesian products $E=U\ti\R^{p|\bq}$, where $U$ is a $\z2$-manifold, with the homogeneity structure defined by the vector field $\n^{p|\bq}$viewed as a vector field on $U\ti\R^{p|\bq}$. Hence, \emph{linear functions} on $E$ are defined, again, as $\n^{p|\bq}$-homogeneous functions of weight 1. It is easy to see that the basic functions on the trivial bundle $E\to U$ are $\n^{p|\bq}$-homogeneous of weight 0. Hence, $\n^{p|\bq}$ acts trivially on pullback functions, and linear functions $f$ on $E$ are of the form
\be\label{lin}
f=\sum_a F_a{x^a}+\sum_\za G_\za{\zx^\za},
\ee
where $F_a, G_\za$ are functions on $U$ (basic functions). The space $\Lin(E)$ of linear functions on $E$ is therefore a free module over the algebra $\cA(U)=\cO_U(|U|)$ of functions on $U$, with free generators $(x^a,\zx^\za)$. The dual free module we call the \emph{module of sections} of $E$ and denote $\Sec(E)$.
\begin{definition}
A \emph{$\z2$-vector bundle of rank $(p|\bq)$} over a $\z2$-manifold $M$ is a $\z2$-fiber bundle $\zp:E\to M$ with the typical fiber $F=\R^{p|\bq}$, equipped with an atlas of local trivialisations
$$t_i : \tU_i=\pi^{-1}(U_i) \stackrel{\sim}{\rightarrow} U_i \times \R^{p|\bq}$$
such that the transition maps
$$t_{ij}=t_i\circ t_j^{-1}:(U_i\cap U_j)\ti\R^{p|\bq}\ra(U_i\cap U_j)\ti\R^{p|\bq}$$
are automorphism of the homogeneity manifolds $(U_i\cap U_j)\ti\R^{p|\bq}$.
\end{definition}
As the local weight vector field is respected by the transition maps, we have on $E$ a globally defined vector field $\n_E$ of degree 0 -- the \emph{weight vector field}. It is also the generator of the multiplication by positive reals in fibers. (Local) \emph{linear functions} on $E$ are $\n_E$-homogeneous functions of weight 1. They form a module $\Lin(E)$ over the algebra of (local) functions on $M$. Local and global sections of $E$ are defined by duality in an obvious way. In this sense elements of $\Lin(E)$ represent global sections of the dual bundle $E^*$ and \emph{vice versa}. Moreover, $M$ is canonically embedded in $E$ as the submanifold on which all linear functions vanish. For any $\bu\in |U|$ we have an obvious identification of the fiber $E_\bu$ with $\mathbf{R}^{p|\bq}$.
\begin{example}
Let $M$ be a $\z2$-manifold of dimension $(p,\bq)$. The \emph{tangent bundle} of $M$ is a $\z2$-vector bundle $\zt_M:\sT M\to M$ of rank $(p,\bq)$, glued from local trivialisations $U_i\ti\R^{p|\bq}$ with local coordinates $\big(x^a_i,\zx^\za_i,\dot x^b_i,\dot\zx^\zb_i\big)$, where $(x^a_i,\zx^\za_i)$ are local coordinates in $U_i$. The gluing transition maps are
\beas&\Phi_{ij}:(U_i\cap U_j)\ti\R^{p|\bq}\ra(U_i\cap U_j)\ti\R^{p|\bq},\\
&\Phi_{ij}\big(x^a_j,\zx^\za_j,\dot x^b_j,\dot\zx^\zb_j\big)=\Big(\zf_{ij}(x^a_j,\zx^\za_j)\,,\pa_{x^b_j}(x^a_i)\cdot\dot x^b_j+\pa_{\zx^\zb_j}(x^a_i)\cdot\dot\zx^\zb_j\,,\pa_{x^b_j}(\zx^\za_i)\cdot\dot x^b_j+\pa_{\zx^\zb_j}(\zx^\za_i)\cdot\dot\zx^\zb_j\Big),
\eeas
where
$$(x^a_i,\zx^\za_i)=\zf_{ij}(x^a_j,\zx^\za_j)$$
are the corresponding transition maps on $M$. The module $\Sec(\sT M)$ of sections of the tangent bundle is nothing but the module of vector fields on $M$. Contrary to the frequent opinion, vector fields are not graded derivations of the algebra of superfunctions, this is valid only for homogeneous vector fields.
\end{example}
\begin{definition}
\emph{Morphisms of $\z2$-vector bundles} are morphisms of the corresponding $\z2$-manifolds which relate the corresponding Euler vector fields; equivalently, the pullbacks of (local) linear functions are (local) linear functions.  A \emph{vector subbundle} of a $\z2$-vector bundle $\zp:E\to M$ is a $\z2$-submanifold $E_0$ such that the Euler vector field is tangent to $E_0$.
\end{definition}
It is easy to see that $M_0=E_0\cap M$ is a $\z2$-submanifold of $M$, and $\zp\,\big|_{E_0}:E_0\to M_0$ is again a
$\z2$-vector bundle. We say that $E_0$ is \emph{covering $M_0$}. If $M_0=M$, we say that the vector subbundle $E_0$ is \emph{full}.
\subsection{Ehresmann connections}\

\mn Let us observe that, for any vector bundle $\zp:E\to M$, the tangent bundle $\zt_E:\sT E\to E$ has another vector bundle structure, this time over $\sT M$. The corresponding Euler vector field is $\dt\n_E$ -- the tangent lift of the Euler vector field on $E$. It is obtained as the generator of the one-parameter group of diffeomorphisms $\sT h_t$ of $\sT E$, where $h_t$ is the multiplication by $e^t$ in $E$. The corresponding vector bundle projection is $\sT\zp:\sT E\to\sT M$. These two vector bundle structures are compatible, i.e., the vector fields $\dt\n_E$ and $\n_{\sT E}$ commute, so we are dealing with a \emph{double vector bundle} (see \cite{Grabowski:2009}).
\begin{definition} Let $P=(P,M,\pi,F)$ be a $\Z_2^n$-fiber bundle, and $\sT\zp:\sT P\to\sT M$ be its tangent prolongation. An \emph{Ehresmann connection} on the fiber bundle bundle $P$ is a \emph{horizontal subbundle} $\cH$ in $\sT P$, i.e., a full $\z2$-vector subbundle $\cH$ of the tangent bundle $\zt_P:\sT P\to P$ such that
$$\sT\zp(\bp)\,\big|_{\cH_\bp}:\cH_\bp\to\sT_{\pi(\bp)} M$$
is an isomorphism of $\z2$-graded vector spaces for all $\bp\in|P|$.

\mn A \emph{linear connection} in a $\z2$-vector bundle $\zp:E\to M$ is an Ehresmann connection $\cH\subset\sT E$ such that $\cH$ is additionally a vector subbundle of the $\z2$-vector bundle $\sT\zp:\sT E\to\sT M$.
In other words, the vector field $\dt{\n_E}$ is tangent to $\cH$.
\end{definition}
\begin{remark}\
\begin{enumerate}
\item Slight adaptation of the proof of Proposition \ref{prop:TransMaps}, we see that the transition functions are morphisms of trivial fibre bundles and that they satisfy the cocycle condition.
\item Note that a $\Z_2^n$-vector bundle is not just a fibre bundle with typical fibre $F = \R^{r|\mathbf{s}}$. The transition functions  must be morphisms of trivial $\Z_2^n$-vector bundles and so linear in fibre coordinates.
\end{enumerate}
\end{remark}

\subsection{Frame bundles }
Let $\zp:E\to M$ be a vector bundle of rank $(p|\bq)$, and let $\tU=\zp^{-1}(U)\simeq U\ti\R^{p|\bq}$ be a local trivialisation with the standard global coordinates $(x^a,\zx^\za)$ on $\R^{p|\bq}$. The module $\Lin(U)$ of linear functions on $\tU$ is a $\z2$-graded and free module over $\cO_M(|U|)$ with homogeneous generators $v=(x^a,\zx^\za)$. Any linear function on $\tU$ is of the form (\ref{lin}). In particular, we have the following.
\begin{proposition}
If $\tv=(\tx^a,\tzx^\za)$ is another set of free generators of the $\z2$-graded and free module over $\cO_M(|U|)$, then
\begin{align}\nn
\tx^a=A^a_bx^b+B^a_\zb\zx^\zb\\
\tzx^\za= C^\za_bx^b+D^\za_\zb\zx^\zb        \nn,
\end{align}
where $A,B,C,D$ are homogeneous functions on $U$ such that
\begin{align}\nn\deg(A^a_b)=0\,,\quad \deg(B^a_\zb)=\deg{\zb}\,, \\
\nn\deg(C^\za_b)=\deg(\za)\,,\quad \deg(D^\za_\zb)=\deg(\za)+\deg(\zb),
\end{align}
and the matrix
$$X_v^\tv=\begin{pmatrix} A & B\\
C & D
\end{pmatrix}
$$
is invertible. In other words, $X_v^\tv:U\to\GL(p,\bq)$.
Moreover, we have $X_v^v=I$ and the cocycle condition
$$X_{v'}^v\cdot X^{v'}_\tv\cdot X^\tv_v=I.$$
\end{proposition}
\no We conclude that any local trivialisation $\tU=\zp^{-1}(U)\simeq U\ti\R^{p|\bq}$ of $E$ identifies the set $\Fr_U$ of all bases of free generators of the module $\Lin(U)$ with the $\z2$-manifold $\GL(p,\bq)(U)$. It is easy to see that these data give rise to a $\GL(p,\bq)$-principal bundle $\Fr$, called the \emph{frame bundle} of $E$. If $E=\sT M$, then this frame bundle we call the \emph{frame bundle of the manifold $M$} and denote $\operatorname{Fr}(M)$.

\section{Concluding remarks}\label{sec:Conc}
In this paper, we have made an initial study of principal $\Z_2^n$-bundles and established some fundamental results. The key message is that the foundations of the classical theory of principal bundles carry over to
$\Z_2^n$-geometry, however, the proofs rely on a lot of fundamental results established in earlier paper see (\cite{Bruce:2020a,Bruce:2020b,Bruce:2018,Bruce:2018b}) as well as heavy use of the functor of points. \par
There are, of course, plenty of open questions and research directions to pursue.  For example, it would be interesting to develop the theory of  principal connections and gauge theory in this setting. It is natural to consider $\z2$-graded extensions of supersymmetry and $\z2$-superfields, so we hope to further explore $\Z_2^n$-geometry and unravel links with novel physics.

\section*{Acknowledgements}
A.J.~Bruce thanks Norbert Poncin for helpful discussions about $\Z_2^n$-geometry and related subjects.   J.~Grabowski acknowledges that his research was funded by the  Polish National Science Centre grant WEAVE-UNISONO under contract number 2023/05/Y/ST1/00043.

\vskip1cm

\noindent Andrew James Bruce\\
{\tt andrewjamesbruce@googlemail.com}\\
https://orcid.org/0000-0001-8197-2263\\
Web of Science Researcher ID: AAQ-4384-2020\\

\noindent Janusz Grabowski\\ \emph{Institute of Mathematics, Polish Academy of Sciences}\\{\small \'Sniadeckich 8, 00-656 Warszawa,
Poland}\\{\tt jagrab@impan.pl}\\
https://orcid.org/0000-0001-8715-2370\\
Web of Science Researcher ID: K-1248-2013

\end{document}